\documentclass[11pt,leqno]{amsart}
\usepackage{amsmath}
\usepackage{amssymb}
\usepackage{amsthm}
\usepackage[dvips]{graphics}
\usepackage{hyperref}
\usepackage{graphicx}
\usepackage[ruled,vlined]{algorithm2e}
\usepackage{multirow}
\usepackage{bbm}
\usepackage{subcaption}

\renewcommand{\arraystretch}{1.2}
\newtheorem{theorem}{Theorem}[section]
\newtheorem{lemma}{Lemma}[section]

\newtheorem{definition}{Definition}[section]

\newtheorem{example}{Example}[section]

\renewcommand{\arraystretch}{1.2}
\usepackage{tikz}
\usetikzlibrary{arrows.meta}
\tikzset{
  slpVar/.style = {draw, rounded corners=2pt, minimum width=14mm, minimum height=6mm, font=\small, fill=blue!8},
  slpOp/.style  = {draw, rounded corners=2pt, minimum width=16mm, minimum height=7mm, font=\small, fill=orange!12},
  slpOut/.style = {draw, rounded corners=2pt, minimum width=28mm, minimum height=7mm, font=\small\bfseries, fill=green!12},
  slpEdge/.style= {->, >=Latex, line width=0.9pt}
}

\def\ba{\begin{array}}
\def\ea{\end{array}}
\def\beq{\begin{equation}}
\def\eeq{\end{equation}}
\def\bea{\begin{eqnarray}}
\def\eea{\end{eqnarray}}
\def\beann{\begin{eqnarray*}}
\def\eeann{\end{eqnarray*}}

\def\R{\mathbb{R}}
\def\S{\mathbb{S}}

\def\inte{\textup{int}}

\def\ln{\textup{ln}}

\newcommand{\ignore}[1]{}

\newcommand{\tx}[1]{\texttt{#1}}
\newcommand{\p}[1]{\pmb{#1}}
\def\inte{\textup{int}}

\title{\bf Efficient Interior-Point Methods for Hyperbolic Programming via Straight-Line Programs}
\author{
Mehdi Karimi \and Levent Tun\c{c}el}
\thanks{\noindent  Mehdi Karimi: Department of Mathematics, Illinois State University, Normal, IL, 61761.  ({e-mail: \bf mkarim3@ilstu.edu}).   Research of this author was supported in part by	 the National
Science Foundation (NSF) under Grant No. CMMI-2347120.\\
 Levent Tun\c{c}el: Department of Combinatorics and Optimization, Faculty of Mathematics,
University
of Waterloo, Waterloo, Ontario N2L 3G1, Canada (e-mail: {\bf 
levent.tuncel@uwaterloo.ca}). Research of this author was supported in part by
a Discovery Grant from the Natural Sciences and
Engineering Research Council (NSERC) of Canada.}

\begin{document}
\begin{abstract}
Hyperbolic (HB) programming generalizes many popular convex optimization problems, including semidefinite and second-order cone programming. Despite substantial theoretical progress on HB programming, efficient computational tools for solving large-scale hyperbolic programs remain limited. 
This paper presents DDS 3.0, a new release of the Domain-Driven Solver, which provides an efficient interior-point implementation tailored for hyperbolic programming. A key innovation lies in a new straight-line program (SLP) representation that enables compact representation and efficient computation of hyperbolic polynomials, their gradients, and Hessians. The SLP structure significantly reduces computational cost, allowing the Hessian to be computed in the same asymptotic complexity as the gradient through a batched reverse-over-forward differentiation scheme. 
We further introduce an improved corrector step for the primal-dual interior-point method, enhancing stability and convergence on convex sets where only the primal self-concordant barrier is efficiently computable. We create a  comprehensive benchmark library  beyond the elementary symmetric polynomials, using several different techniques. Numerical experiments demonstrate substantial performance gains of DDS 3.0 compared to first-order Frank–Wolfe algorithm, homotopy method, and SDP software utilizing SDP relaxations.

\end{abstract}
\maketitle

\pagestyle{myheadings} \thispagestyle{plain}
\markboth{KARIMI and TUN{\c C}EL}
{Hyperbolic Programming}

%%%%%%%%%%%%%%%%%%%%%%%
\section{Introduction} \label{sec:hyper}
Convex optimization forms a major part of the foundation of modern optimization theory and machine learning. Among its many subclasses, hyperbolic programming stands out as one of the most general and elegant frameworks, involving constraints defined by \emph{hyperbolic (HB) polynomials}. A \emph{homogeneous} polynomial $p(x) : \R^m \rightarrow \R$ (every monomial has the same degree $d$) is called a  \emph{hyperbolic polynomial in the direction $\p{e} \in \R^m$} if 
\begin{itemize}
\item $p(\p{e})>0$. 
\item for every $\p{x} \in \R^m$, the univariate polynomial $p(\p{x}-t\p{e})$ has only real roots. 
\end{itemize}
For a hyperbolic polynomial $p$ and a given direction $\p{e}$, the roots of $p(\p{x}-t\p{e})$ are called the \emph{eigenvalues} of $\p{x}$, for a given $\p{x} \in \R^m$. 
We define the \emph{characteristic map} $\lambda: \R^m \rightarrow \R^d$ with respect to $p$ and $\p{e}$ such that 
for every $\p{x} \in \R^m$, the elements of $\lambda(\p{x})$ are the roots of the univariate polynomial $p(\p{x}-t\p{e})$, in non-increasing order $\lambda_1(\p{x}) \geq \lambda_2(\p{x}) \geq \cdots \geq \lambda_d(\p{x})$.  
The underlying \emph{hyperbolicity cone} is defined as
\begin{eqnarray} \label{eq:hyper-1}
\Lambda_+(p,\p{e}):=\{\p{x} \in \R^m: \lambda(\p{x}) \geq 0\}. 
% \Lambda(p,e):=\{x \in \R^m: p(x+\lambda e)  \geq 0, \forall \lambda \geq 0\}. 
\end{eqnarray}
To see some examples, by choosing the HB polynomial $p(\p{x}):=x_1^2-x_2^2-\cdots-x_m^2$ in the direction of $\p{e}:=(1,0,\ldots,0)^\top$, we see that the second-order cone is a hyperbolicity cone. By choosing the HB polynomial $p(X) := \det(X)$ defined on $\S^n$ (the set $n$-by-$n$ symmetric matrices)  in the direction of $I$, we see that the positive semidefinite cone is a hyperbolicity cone. 
\emph{Hyperbolic programming} is a convex optimization problem that involves  a hyperbolicity cone constraint. HB programs are one of the largest classes of well-structured convex optimization problems. We just saw that the second-order cone programming (SOCP) and semidefinite programming (SDP) are special cases of HB programming. 

Hyperbolic polynomials have been studied at least since the 1930's
(going back to the work of Petrovsky \cite{Petrovsky1937}).  Starting in the early 1990's
there has been a renewed interest stemming from combinatorics
as well as optimization.  This increased amount of activity
allowed the subject to branch into various research areas
including systems and control theory, operator theory (see,
for instance, the Marcus-Spielman-Srivastava \cite{MSS2015} solution of
Kadison-Singer problem), interior-point methods (see, for instance,
\cite{guler1997hyperbolic,renegar2006hyperbolic,NT2016,DTV2025} and the references therein),
discrete optimization and combinatorics (see, for instance,
\cite{Gurvits2008}, \cite{Wagner2011}, \cite{AOV2018} and references therein),
semidefinite programming and semidefinite representations, matrix
theory as well as theoretical computer science.

Helton-Vinnikov Theorem~\cite{Vinnikov1993,HV2007,LPR2005} implies that
all three dimensional hyperbolicity
cones are spectrahedra and every hyperbolic
polynomial giving rise to a 3-dimensional hyperbolicity cone admits a
very strong determinantal representation.
There are many generalizations of Helton-Vinnikov theorem (see
\cite{SV2018} and the references therein),
counter examples to certain proposed generalizations of Helton-Vinnikov
Theorem (see \cite{branden2011obstructions} and the references
therein), various spectrahedral and spectrahedral-shadow representations
for interesting hyperbolicity cones (see \cite{NS2015},
\cite{K2021} and the references therein).
HP programming appears in many other areas such as control theory and combinatorial optimization \cite{gurvits2006hyperbolic,borcea2009lee}

One of the appealing features of hyperbolic programming is the potential for designing efficient interior-point methods \cite{interior-book}, since every hyperbolicity cone is equipped with a natural self-concordant (s.c.) barrier function. 
 
\begin{theorem}[G{\"u}ler \cite{guler1997hyperbolic}] Let $p(\p{x})$ be a homogeneous polynomial of degree $d$, which is hyperbolic in direction $\p{e}$. Then, the function $-\ln(p(\p{x}))$ is a $d$-logarithmicaly homogenous s.c.\ barrier  for $\Lambda_+(p,\p{e})$. 
\end{theorem}

After decades of successful implementation of interior-point algorithms for optimization over symmetric cones (see, for instance, \cite{toh1999sdpt3,sedumi,mosek,DDS}), there have been several recent efforts to create efficient software for handling other convex sets with available computationally efficient self-concordant (s.c.) barriers. Some modern interior-point software for solving convex optimization problems (beyond symmetric cones) using s.c.\ barriers are: a MATLAB-based software package Alfonso (\cite{alfonso}), a software package Hypatia in the Julia language (\cite{coey2022solving}), Clarabel,  available in either  Julia or Rust \cite{Clarabel_2024}, and a MATLAB-based software package DDS (\cite{DDS}). While Hypatia and Clarabel implement barrier-based algorithms for general convex cones, neither provides built-in support for user-defined hyperbolic polynomials. DDS remains the only interior-point solver with explicit HB programming support. Nevertheless, the rapid development of theoretical results on hyperbolic programming highlights the need for an efficient solver which can be used to verify many hypotheses numerically. In this direction, the authors of \cite{nagano2024projection} studied the problem of projecting a point onto an arbitrary hyperbolicity cone and proposed a Frank-Wolfe method for solving it. They compared their approach with DDS 2.2. However, because the straight-line program representation of polynomials was not yet fully supported in earlier versions of DDS, they were restricted to using the monomial-based implementation. This significantly limited the scalability of their experiments, forcing the benchmarking to be conducted only on relatively small instances.

A common computational misunderstanding is that, if the generalized Lax conjecture (see Section \ref{sec:SDP}) holds, then hyperbolic (HB) programming becomes unnecessary because every such problem could be reformulated and solved more efficiently as a semidefinite program (SDP). However, the Lax conjecture does not provide any guarantee on the size of the matrices in \eqref{eq:Lax}, and in fact, these matrices can be much larger than $m$. As a result, even if a HB program can in principle be represented as an SDP, the corresponding SDP may be prohibitively large for current state-of-the-art SDP solvers. This observation highlights the importance of developing dedicated HB programming solvers. Now that the theory and tools for interior-point algorithms are well-established, it is crucial to provide an efficient baseline solver for HB programming that can serve as a point of comparison with alternative approaches \cite{nagano2024projection, klingler2025homotopy}.

This work addresses these computational limitations by introducing algorithmic and structural innovations in DDS 3.0 that make large-scale hyperbolic programming more tractable in practice.
\subsection{Contributions}
\begin{itemize}
\item {\bf Efficient representation, evaluation, and differentiation of hyperbolic barriers:}
We introduce a systematic and efficient implementation for polynomials using straight-line programs (SLP), a compact representation which avoids explicit monomial enumeration. Within this framework, we derive and implement efficient procedures for computing both the gradient and the Hessian. In particular, we show that evaluating the gradient of the barrier function has the same computational complexity as evaluating the function. Moreover, by exploiting vectorized ``bulk" operations, we can evaluate the Hessian with the same order of ``vectorized" operations for evaluating the function.
\item {\bf A new corrector step for interior-point methods:}
We design and implement a novel corrector step within the primal-dual interior-point method in DDS. This step significantly improves the efficiency of the solver for the cones for which only the primal s.c.\ barrier is computationally available. This development enhances the robustness of DDS on challenging instances.
\item {\bf DDS 3.0 release with extended SLP support:} We release DDS 3.0, which incorporates the methodological and computational advances introduced in this work. The new version includes an extended package for handling hyperbolic polynomials represented via SLP, with efficient implementations for the first- and second-order derivative computations. We present extensive numerical results on the newly created benchmark library, to show the performance improvements over the previous version of DDS and other new algorithms such as the Frank-Wolfe algorithm designed in \cite{nagano2024projection}. 
\item {\bf Benchmark library of hyperbolic polynomials:}
Beyond elementary symmetric polynomials (ESP), which are common in the literature, we extend the SLP representation to a large family of hyperbolic polynomials. Several methods are used to create hyperbolic polynomials other than ESPs: using matroids, directional derivatives, products of linear forms, and compositions. We have created an extended benchmark library of hyperbolic programming instances.
\end{itemize}

%DDS handles optimization problems involving hyperbolic polynomials using the above s.c.\ barrier. A computational problem is that, currently, we do not have a practical, efficient, algorithm to evaluate the LF conjugate of $-\ln(p(x))$. Therefore, DDS uses a primal-heavy version of the algorithm for these problems. 

\section{Different formats for Hyperbolic polynomials} \label{sec:poly-format}
Hyperbolic polynomials can be represented in three main forms, all of which are implemented in DDS. In this section, we describe these representations and explain how to define the variable \texttt{poly} in DDS, which encodes the polynomial structure. Among these representations, the straight-line program (SLP) format offers the highest computational efficiency and scalability, making it particularly suitable for a wide range of hyperbolic programming instances.

\noindent{\bf Using monomials:} This representation defines a polynomial by listing all the monomials. In DDS, if $p(\p{x})$ is a polynomial of $m$ variables, then \tx{poly} is an $k$-by-$(m+1)$ matrix, where $k$ is the number of monomials. In the $j$th row, the first $m$ entries are the powers of the $m$ variables in the monomial, and the last entry is the coefficient of the monomial in $p(\p{x})$. For example, if $p(\p{x})=x_1^2-x_2^2-x_3^2$, then
\begin{eqnarray*}
\tx{poly}:=\left[\begin{array}{cccc} 2&0&0&1 \\ 0&2&0&-1 \\ 0&0&2&-1 
\end{array}	\right].
\end{eqnarray*}
In many applications, the above matrix is very sparse. Moreover, for many hyperbolic polynomials, the number of monomials grows very fast versus the number of variables and make the representation computationally intractable. 

\noindent{\bf Straight-line program data structure.}  
A \emph{straight-line program} (SLP) is an algebraic structure used to represent composite functions
as a sequence of arithmetic operations.
Each intermediate quantity is computed from previously computed ones, forming a \emph{computational graph}.
This concept is closely related to those used in automatic differentiation
and neural network computations \cite{griewank2008evaluating, goodfellow2016deep}.
The structure can be visualized as a directed graph, where nodes correspond to arithmetic
operations and edges represent data dependencies.
SLPs are widely used in symbolic computation, automatic differentiation,
and optimization algorithms due to their efficiency in evaluating both the function
and its derivatives through forward and reverse propagation.

In our implementation, each node of the straight-line program (SLP) is stored as an element of a MATLAB \texttt{struct} array. This format makes the representation both human-readable and easily extensible compared to the matrix-based encoding that was used in DDS 2.2 \cite{DDS}.  

Each node is represented by four fields:
\[
\texttt{poly}(k).\texttt{id}, \qquad 
\texttt{poly}(k).\texttt{op}, \qquad 
\texttt{poly}(k).\texttt{in}, \qquad 
\texttt{poly}(k).\texttt{coef}.
\]

\begin{itemize}
\item \texttt{id}: Actual identifier of the node in the computational graph that is used to reference to the node. 
  \item \texttt{op}: a string, identifying the operation, chosen from
  \[
  \{\texttt{`add'}, \texttt{`sub'}, \texttt{`mul'}, \texttt{`pow'}\}.
  \]
  \item \texttt{in}: a vector of two integers giving the indices of the input nodes. For example, \texttt{[2,3]} means that this operation takes as inputs the values computed at nodes $f_2$ and $f_3$. Constants are represented by index $0$, corresponding to $f_0=1$.
  \item \texttt{coef}: a scalar (or a row vector, in the batched case) multiplying the result of the operation. This allows scaling factors to be encoded directly at each node.
\end{itemize}

Formally, assume $p(\p{x})$ has $m$ variables. We define
\[
f_0 := 1, \qquad f_i := x_i \quad (i \in \{1,\ldots,m\}).
\]
Each subsequent node $f_\ell$ is generated by an operation of the form
\[
f_\ell \;=\; \alpha \,\big(f_{i} \;\Box\; f_{j}\big),
\]
where $\Box$ is one of $\{+,-,\times,\text{power}\}$ and $\alpha$ is the coefficient stored in \texttt{coef}. The structure entry for $f_\ell$ contains the following fields:
\[
\{\texttt{id}=\ell, \texttt{op}=\Box,\,\texttt{in}=[i,j],\,\texttt{coef}=\alpha\}.
\]

\medskip
\noindent{\bf Example.}  
For $p(\p{x})=x_1^2 - x_2^2 - x_3^2$, the struct array is:
\[
\begin{array}{ll}
\texttt{poly}(1): & \{\texttt{id} = 4,\,\texttt{op}=\text{'mul'},\, \texttt{in}=[1,1],\, \texttt{coef}=1\}, \\[0.3em]
\texttt{poly}(2): & \{\texttt{id} = 5,\,\texttt{op}=\text{'mul'},\, \texttt{in}=[2,2],\, \texttt{coef}=-1\}, \\[0.3em]
\texttt{poly}(3): & \{\texttt{id} = 6,\,\texttt{op}=\text{'mul'},\, \texttt{in}=[3,3],\, \texttt{coef}=-1\}, \\[0.3em]
\texttt{poly}(4): & \{\texttt{id} = 7,\,\texttt{op}=\text{'add'},\, \texttt{in}=[4,5],\, \texttt{coef}=1\}, \\[0.3em]
\texttt{poly}(5): & \{\texttt{id} = 8,\,\texttt{op}=\text{'add'},\, \texttt{in}=[6,7],\, \texttt{coef}=1\}.
\end{array}
\]

This representation corresponds exactly to the computation graph shown in Figure~\ref{fig:slp-struct-dag}.  
\begin{figure}
\centering
\includegraphics[scale=.7]{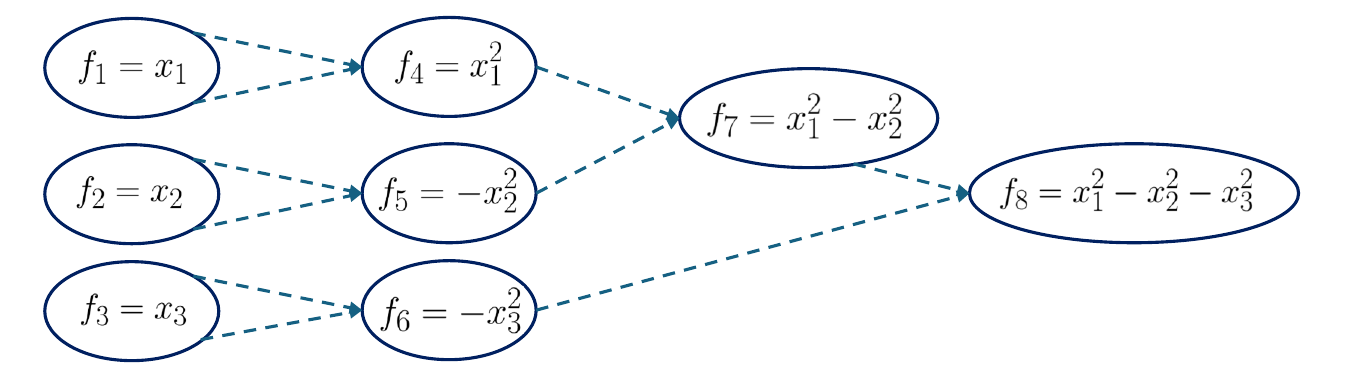}
\caption{The SLP graph structure for the function $p(\p{x})=x_1^2 - x_2^2 - x_3^2$ \label{fig:slp-struct-dag}}
\end{figure}

 Unlike the matrix form in DDS 2.2 \cite{DDS}, which encoded operations using numerical opcodes (11, 22, 33, 44), the struct form directly names operations, improving readability and reducing the risk of indexing errors. The representation is extensible and new nodes can be added as new fields without changing the global structure. The struct-based format integrates naturally with batched evaluation.

\noindent{\bf Determinantal representation:} In this case, if possible, the polynomial $p(\p{x})$ is written as
\begin{eqnarray} \label{eq:hyper-3}
p(\p{x})=\det(H_0+x_1H_1+x_2 H_2+\cdots +x_mH_m),
\end{eqnarray}
where $H_i$, $i \in \{0,1,\ldots,m\}$ are symmetric matrices.  This format can be input into DDS by defining
\[
\tx{poly}:=[\tx{sm2vec}(H_0) \ \ \tx{sm2vec}(H_1) \ \cdots \ \tx{sm2vec}(H_m)].
\]
For example, for $p(\p{x})=x_1^2-x_2^2-x_3^2$, we can have
\[
H_0:=\begin{pmatrix} 0&0\\0&0\end{pmatrix}, \ \ 
H_1:=\begin{pmatrix} 1&0\\0&1\end{pmatrix}, \ \
H_2:=\begin{pmatrix} 1&0\\0&-1\end{pmatrix}, \ \
H_3:=\begin{pmatrix} 0&1\\1&0\end{pmatrix}. \ \
\]

\section{Derivative Evaluations via Straight-Line Programs (SLP)}
Even though the straight-line program (SLP) representation allows us to express polynomials efficiently, a robust and high-performance interior-point solver also requires efficient computation of their gradients and Hessians. The improvements introduced in DDS 3.0 rely heavily on these optimized derivative implementations, which form the computational backbone of the solver. In this section, we describe the techniques used to compute gradients and Hessians within the SLP framework and discuss how vectorized and bulk operations are leveraged to achieve significant speedups.
\subsection*{Gradient Computation}
For a function $f(\p{x})$ represented as a SLP, the gradient \( \nabla f(\p{x}) \in \mathbb{R}^m \) can be computed using \emph{reverse-mode automatic differentiation (AD)}, as described in Algorithm~\ref{alg:slp_grad}. Let \( f_k \) denote each intermediate value. We define:
\[
\text{adjoint}(f_k) := \frac{\partial f}{\partial f_k}
\]
and initialize:
\[
\text{adjoint}(f_{\text{output}}) := 1.
\]
Then, we traverse the SLP in reverse, applying the chain rule at each operation:
\begin{itemize}
  \item Addition: \( f_k = \alpha(f_i + f_j) \Rightarrow \frac{\partial f}{\partial f_i} \mathrel{+}= \alpha \cdot \text{adjoint}(f_k), \quad \frac{\partial f}{\partial f_j} \mathrel{+}= \alpha \cdot \text{adjoint}(f_k) \)
  \item Multiplication: \( f_k = \alpha(f_i \cdot f_j) \Rightarrow \frac{\partial f}{\partial f_i} \mathrel{+}= \alpha f_j \cdot \text{adjoint}(f_k), \quad \frac{\partial f}{\partial f_j} \mathrel{+}= \alpha f_i \cdot \text{adjoint}(f_k) \)
  \item Power: \( f_k = \alpha f_i^p \Rightarrow \frac{\partial f}{\partial f_i} \mathrel{+}= \alpha p f_i^{p-1} \cdot \text{adjoint}(f_k) \)
\end{itemize}
where $a \mathrel{+}= b$ means add $b$ to $a$ and store the results back in $a$. The final gradient is:
\[
\nabla f(\p{x}) = \left[ \frac{\partial f}{\partial x_1}, \ldots, \frac{\partial f}{\partial x_m} \right]^\top
= \left[ \text{adjoint}(f_1), \ldots, \text{adjoint}(f_m) \right]^\top.
\]

\begin{algorithm}[H]
\caption{Gradient computation via reverse-mode AD for an SLP}
\label{alg:slp_grad}
\KwIn{
\begin{tabular}{ll}
$\p{x} \in \mathbb{R}^m$: & input vector \\
$\{(\alpha_k, i_k, j_k, \mathrm{op}_k)\}_{k=1}^N$: & SLP description \\
%& $\alpha_k$: coefficient for operation $k$ \\
%& $i_k, j_k$: operand indices \\
%& $\mathrm{op}_k \in \{\texttt{add}, \texttt{sub}, \texttt{mul}, \texttt{pow}\}$
\end{tabular}
}
% \KwOut{$\nabla f(\mathbf{x}) \in \mathbb{R}^n$}

\BlankLine
\textbf{Forward pass:} \\
Initialize $f_0 \gets 1$, $f_t \gets x_t$ for $t=1,\dots,n$\;
\For{$k \gets 1$ \KwTo $N$}{Calculate the value of $f_{n+k}$  }

\BlankLine
\textbf{Reverse pass:} \\
Initialize $d_t \gets 0$ for all $t$; set $d_{m+N} \gets 1$\;
\For{$k \gets N$ {\bf down to}  $1$}{
    Let $u \gets i_k$, $v \gets j_k$, $\delta \gets d_{n+k}$\;
    \uIf{$\mathrm{op}_k = \texttt{add}$}{
        $d_u \gets d_u + \alpha_k \cdot \delta$; \quad
        $d_v \gets d_v + \alpha_k \cdot \delta$
    }
    \uElseIf{$\mathrm{op}_k = \texttt{sub}$}{
        $d_u \gets d_u + \alpha_k \cdot \delta$; \quad
        $d_v \gets d_v - \alpha_k \cdot \delta$
    }
    \uElseIf{$\mathrm{op}_k = \texttt{mul}$}{
        $d_u \gets d_u + \alpha_k \cdot f_v \cdot \delta$; \quad
        $d_v \gets d_v + \alpha_k \cdot f_u \cdot \delta$
    }
    \uElseIf{$\mathrm{op}_k = \texttt{pow}$}{
        $p \gets v$; \quad
        $d_u \gets d_u + \alpha_k \cdot p \cdot f_u^{\,p-1} \cdot \delta$
    }
}

\BlankLine
\Return{$\nabla f(\p{x}) = (d_1, d_2, \dots, d_m)^\top$}
\end{algorithm}

\subsection*{Hessian Computation}

Let \( f : \mathbb{R}^m \to \mathbb{R} \) be a polynomial represented by a straight-line program (SLP). The Hessian \( H = \nabla^2 f(\p{x}) \in \mathbb{R}^{m \times m} \) is computed using a \emph{reverse-over-forward} strategy, which combines forward-mode and reverse-mode automatic differentiation (AD), as detailed in Algorithm \ref{alg:slp_hess}.

First, the gradient of every intermediate variable \( f_k \) with respect to the input \( x \) is computed and stored. This is done using forward-mode differentiation as the SLP is evaluated. Specifically, for each node \( f_k \), we compute $
\mathbf{g}_k := (\nabla f_k)^\top \in \mathbb{R}^{1 \times m}$.
These rows are assembled into a matrix:
$
\texttt{MG} \in \mathbb{R}^{N \times m},
$
where \( N \) is the number of total nodes in the SLP. The matrix \texttt{MG} captures the derivative of all intermediate values with respect to all inputs.

\begin{algorithm}[H]
\caption{Batched Hessian Computation via Reverse-over-Forward AD for an SLP \label{alg:slp_hess}}
\KwIn{SLP representation \(\mathcal{P}\) with $N$ nodes, input \(x \in \mathbb{R}^m\)}
\KwOut{Hessian matrix \(H \in \mathbb{R}^{m \times m}\)}

\BlankLine
\textbf{Forward sweep (primal + forward-mode):}
\begin{enumerate}
    \item Evaluate all intermediate values \(f_k\) in the SLP at \(x\) (primal evaluation).
    \item Simultaneously compute the gradient of all intermediate values:
    \[
    \texttt{MG}(k,:) \gets (\nabla f_k(\p{x}))^\top \quad k=1,\dots,N
    \]
\end{enumerate}

\BlankLine
\textbf{Batched reverse sweep:}
\begin{enumerate}
    \item Initialize \(\texttt{DF} \gets \mathbf{0}_{N \times m}\) (adjoints for all intermediates, one column per seed).
    \item Set \(\texttt{DF}(\text{output index},:) \gets I_n\) (seeds for reverse-mode).
    \item For each operation \(k = N, N-1, \dots, 1\) in reverse order:
    \begin{enumerate}
        \item Read operands \((i,j,op)\) and coefficient \(\alpha\) from \(\mathcal{P}\).
        \item Update \(\texttt{DF}(i,:)\) and \(\texttt{DF}(j,:)\) according to the derivative rules of the operation \(op\), using batched multiplication with \(\texttt{MG}\) for forward derivatives.
    \end{enumerate}
\end{enumerate}

\BlankLine
\textbf{Extract Hessian:}
\[
H \gets \texttt{DF}(1:m,:)
\]
\[
H \gets \frac{H + H^\top}{2} \quad \text{(symmetrize for numerical stability)}.
\]

\end{algorithm}

To compute the Hessian \( H \), we use the identity:
\[
\frac{\partial^2 f}{\partial x_i \partial x_k} = \frac{\partial}{\partial x_i} \left( \frac{\partial f}{\partial x_k} \right).
\]
To perform this column by column, for each \( k = 1, \dots, m \), we treat the directional derivative \( \partial f / \partial x_k \) as a scalar function \( g_k(x) \), and apply reverse-mode AD to compute:
\[
H_{:,k} = \nabla g_k(x) = \nabla \left( \frac{\partial f}{\partial x_k} \right).
\]

\subsubsection*{Batched Reverse Pass.}  
To avoid repeated reverse-mode sweeps, we use a batched implementation in DDS 3.0. Instead of seeding reverse-mode AD with one direction at a time, we pass the entire matrix \( \texttt{MG} \in \mathbb{R}^{N \times m} \) and perform reverse-mode backpropagation simultaneously over all directions:
\[
H = \left[ \nabla \left( \frac{\partial f}{\partial x_1} \right) \mid \cdots \mid \nabla \left( \frac{\partial f}{\partial x_m} \right) \right].
\]
This yields the full Hessian in one pass. 
In the column-by-column method, the reverse-mode sweep is repeated \( n \) times, each with a different seed vector. 
A batched implementation would accept the full seed matrix \( \texttt{MG} \) as input and perform a single reverse-mode sweep to compute all Hessian columns simultaneously.
This avoids redundant passes over the same computation graph and can significantly reduce runtime when \( n \) is large. The computational complexity of the gradient and Hessian calculations are summarized in the following theorem. 

\begin{theorem}
Assume that $p(\p{x}): \R^m \rightarrow \R$ is a polynomial with a straight-line program with a graph on $N$ nodes. Then, the gradient of $p$ can be calculated in $\mathcal O(N)$ operations. Moreover, the Hessian of $p$ can be calculated in $\mathcal O(mN)$ operations or equivalently  $\mathcal O(N)$ vector operations on $m$-dimensional vectors.  
\end{theorem}
\begin{proof}
By inspecting Algorithm~\ref{alg:slp_grad}, we observe that the gradient computation consists of one forward and one reverse sweep through the straight-line program, each involving $N$ nodes. Since each node performs a constant number of scalar operations, the total computational cost is $O(N)$.

For the Hessian computation in Algorithm~\ref{alg:slp_hess}, using the batched reverse-mode differentiation, we again perform two passes of length $N$. In this case, however, the intermediate adjoint values are vectors of dimension $m$, corresponding to the number of input variables. Therefore, each step involves $\mathcal O(m)$ arithmetic operations. Consequently, the overall complexity of the Hessian computation is $\mathcal O(N)$ vector operations of size $m$, which amounts to $\mathcal O(mN)$ operations.
\end{proof}
Modern numerical environments such as MATLAB and Julia provide highly optimized and parallelized routines for vector and matrix operations. When implemented using bulk operations, the gradient and Hessian evaluations in the SLP framework achieve nearly the same computational complexity, and often the same practical runtime, as the function evaluation itself.

\section{Hyperbolic Polynomials as Straight Line Programs} \label{SLP}

\subsection{Recursive Formula for Elementary Symmetric Polynomials}

Let \( \p{x} = (x_1, x_2, \ldots, x_n)^\top \in \mathbb{R}^n \), and let \( e_k^{(i)} \) denote the degree-\( k \) elementary symmetric polynomial of the first \( i \) variables \( x_1, \ldots, x_i \). That is,
\[
e_k^{(i)} = \sum_{1 \leq j_1 < j_2 < \dots < j_k \leq i} x_{j_1} x_{j_2} \cdots x_{j_k},
\]
with the conventions:
\[
e_0^{(i)} := 1 \quad \text{for all } i \geq 0, \qquad \text{and} \qquad e_k^{(i)} := 0 \quad \text{if } k > i \text{ or } k < 0.
\]
The recursive algorithm constructs each \( e_k^{(i)} \) using two previously computed values,
\[
e_k^{(i)} \;:=\; e_k^{(i-1)} \;+\; x_i \cdot e_{k-1}^{(i-1)}.
\]
To efficiently implement this recursion in a straight-line program (SLP) representation,
we employ a dynamic programming approach that stores the index of each computed \( e_{k'}^{(i')} \) in a lookup table (e.g., a two-dimensional map or dictionary) so that every intermediate value is computed exactly once and can be referenced in constant time.

The algorithm proceeds column-by-column in \(i\) (corresponding to the number of variables considered), and within each column, row-by-row in \(k\) (corresponding to the polynomial degree).  
For each pair \((i,k)\), the SLP node for \( e_k^{(i)} \) is created by:
(1) retrieving the index of \( e_k^{(i-1)} \) from the table,  
(2) retrieving the index of \( e_{k-1}^{(i-1)} \),  
(3) creating a multiplication node for \( x_i \cdot e_{k-1}^{(i-1)} \), and  
(4) creating an addition node to sum it with \( e_k^{(i-1)} \).  
The index of the new node is then stored back in the table for later reuse.
Since each of the \(n\) columns has at most \(k\) rows, the total number of SLP operations (additions and multiplications) required to compute \( e_k^{(n)} \) is \(\mathcal{O}(nk)\).
This is exponentially more efficient than enumerating all \( \binom{n}{k} \) monomials directly.
Moreover, the explicit indexing structure in the dynamic programming table ensures that the SLP representation is compact and reuses intermediate results, reducing both computation time and memory usage.

\subsection{Hyperbolic polynomials from matriods} Some interesting classes of hyperbolic polynomials can be constructed using \emph{matroids}. Let us first define a matroid.
\begin{definition}
A \emph{ground set} $E$ and a collection of \emph{independent sets } $\mathcal I \subseteq 2^E$ form a \emph{matroid} $M=(E,\mathcal I)$ if all of the following hold:
\begin{itemize}
\item $\emptyset \in \mathcal I$, 
\item if $A \in \mathcal I$ and $B \subset A$, then $B \in \mathcal I$, 
\item if $A,B \in \mathcal I$, and $|B| > |A|$, then there exists $b \in B \setminus A$ such that $A \cup \{b\} \in \mathcal I$. 
\end{itemize}
\end{definition}
The independent sets with maximum cardinality are called the \emph{bases} of the matroid, we denote the set of bases of the matroid by $\mathcal B$. We can also assign a \emph{rank} function $r_M: 2^E \rightarrow \mathbb Z_+$ as  $r_M(A) := \max \{|B| : B \subseteq A, B \in \mathcal I\}$.
Naturally, the rank of a matroid is the cardinality of any basis of it.
The \emph{basis generating polynomial} of a matroid is defined as
\begin{eqnarray} \label{eq:hyper-6}
p_M(\p{x}) := \sum_{B \in \mathcal B}  \  \prod_{i \in B}  x_i.
\end{eqnarray}
A class of hyperbolic polynomials are the basis generating polynomials of certain matroids with the \emph{half-plane property} \cite{choe2004homogeneous}. A polynomial has the half-plane property if it is nonvanishing whenever all the variables lie in the open right half-plane. We state it as the following lemma:
\begin{lemma}
Assume that $M$ is a matroid with the half-plane property. Then its basis generating polynomial is hyperbolic in any direction in the positive orthant. 
\end{lemma}

Several classes of matroids have been proven to have the half-plane property \cite{choe2004homogeneous,branden2007polynomials,AminiBranden2018}. As the first example, we introduce the Vamos matroid. The ground set has size 8 (can be seen as 8 vertices of the cube) and the rank of the matroid is 4. All subsets of size 4 are independent except 5 of them.  Vamos matroid has the half-plane property \cite{wagner2009criterion}. The basis generating polynomial is:
\begin{eqnarray*} \label{eq:hyper-8}
 p_V(\p{x}):= 
 e_4^8(x_1,\ldots,x_8) - x_1x_2x_3x_4-x_1x_2x_5x_6-x_1x_2x_6x_8-x_3x_4x_5x_6-x_5x_6x_7x_8,
\end{eqnarray*}	
where $e_d^m(x_1,\ldots,x_m)$  is the elementary  symmetric polynomial of degree $d$ with $m$ variables. Note that elementary symmetric polynomials are hyperbolic in the direction of the vector of all ones.  Some extensions of the Vamos matroid also satisfy the half-plane property \cite{burton2014real}. These matroids give the following Vamos-like basis generating polynomials:
\begin{eqnarray} \label{eq:hyper-9}
&& p^m_{VL}(\p{x}):=  \nonumber \\
&& e_4^{2m}(x_1,\ldots,x_{2m}) - \left(\sum_{i=2}^m x_1x_2x_{2k-1}x_{2k} \right) - \left(\sum_{i=2}^{m-1} x_{2k-1}x_{2k}x_{2k+1}x_{2k+2}\right).
\end{eqnarray}
These polynomials have an interesting property of being counterexamples to one generalization of the	Lax conjecture \cite{branden2011obstructions,burton2014real}. 
%Explicitly, there is no power $k$ and symmetric matrices $H_2, \ldots, H_{2m}$ such that  
%\begin{eqnarray} \label{eq:hyper-13}
%(p_{VL}(x))^k = \det(x_1I+x_2H_2+\cdots+x_{2m}H_{2m}).
%\end{eqnarray}
We construct a straight-line program  for the Vamos-like polynomial by starting from the elementary symmetric polynomial \(e_4^{(2m)}(x_1,\ldots,x_{2m})\). We showed that this polynomial can be constructed in \(\mathcal{O}(m)\) operations using the recursive formula. The Vamos-like polynomial $p^m_{VL}$
is then obtained by subtracting two structured collections of quartic monomials. In the SLP, we first create and reuse the pairwise products \(y_k := x_{2k-1}x_{2k}\) for \(k=1,\ldots,m\). The first correction is accumulated as \(\sum_{k=2}^{m} (y_1\,y_k)\) and the second as \(\sum_{k=2}^{m-1} (y_k\,y_{k+1})\), each via a linear number of multiply/add nodes. Finally, these two sums are subtracted from the output node of the \(e_4^{(2m)}\) block. This SLP reuses the shared intermediates \(y_k\) and thus encodes the Vamos-like polynomial in \(\mathcal{O}(m)\) additional operations beyond the \(e_4^{(2m)}\) block.

Graphic matroids also have the half-plane property \cite{choe2004homogeneous}. However, the hyperbolicity cones arising
from graphic matroids are isomorphic to positive semidefinite cones. This can be
proved using a classical result that the characteristic polynomial
of the bases of graphic matroids is the determinant of the Laplacian
(with a row and a column removed) of the underlying graph \cite{branden2014hyperbolicity}.
Consider a graph $G=(V,E)$ and let $\mathcal T$ be the set of all spanning trees of $G$. Then the generating polynomial defined as 
\begin{eqnarray} \label{eq:hyper-10}
T_G(\p{x}) := \sum_{T \in \mathcal T}  \  \prod_{e \in T}  x_e
\end{eqnarray}
is a hyperbolic polynomial. Several matroid operations preserve the half-plane property as proved in \cite{choe2004homogeneous}, including taking minors,
duals, and direct sums.

\subsection{Derivatives of hyperbolic polynomials}
Hyperbolicity of a polynomial is preserved under directional derivative operation (see \cite{renegar2006hyperbolic}):
\begin{theorem}  \label{thm:hyper-1}
Let $p(\p{x})\in \R[x_1,\ldots,x_m]$ be hyperbolic in direction $\p{e}$ and of degree at least 2. Then the polynomial 
\begin{eqnarray} \label{eq:hyper-11}
p'_{\p{e}}(\p{x}) := (\nabla p)(x)[\p{e}]
\end{eqnarray}
is also hyperbolic in direction $\p{e}$. Moreover, 
\begin{eqnarray} \label{eq:hyper-12}
\Lambda(p, \p{e})  \subseteq \Lambda (p'_{\p{e}}, \p{e}).
\end{eqnarray}
\end{theorem}
\noindent DDS 3.0 package comes with an internal function \\
\tx{slp\_directional\_derivative(poly,e)} \\
where given a polynomial in SLP and a direction $\p{e}$, returns the directional derivative as a SLP. 
\subsection{Product of linear forms}
Assume that $p(\p{x}):= l_1(\p{x}) \cdots l_\ell(\p{x})$, where $l_1(\p{x}),\ldots,l_\ell(\p{x})$ are linear forms. Then, $p$ is hyperbolic in the direction of any vector $\p{e}$ such that $p(\p{e}) \neq 0$. As an example, consider $e_m^m(x_1,\ldots,x_m) = x_1 \cdots x_m$. 
Recursively applying Theorem \ref{thm:hyper-1} to such polynomials $p$ leads to many hyperbolic polynomials, including all elementary symmetric polynomials. For some properties
of hyperbolicity cones of elementary symmetric polynomials, see \cite{Zinchenko2008}. For some preliminary computational experiments on such hyperbolic programming problems, see \cite{Myklebust2015}. 
For the polynomials constructed by the products of linear forms, it is more efficient to use the straight-line program. 
For a polynomial $p(\p{x}):= l_1(\p{x}) \cdots l_\ell(\p{x})$, let us define an $\ell$-by-$m$ matrix $L$ where the $j$th row contains the coefficients of $l_j$. In DDS,  we have created a function \\
 \tx{poly = lin2slp(M,d)} \\
\noindent which returns a polynomial in SLP form  as the product of linear forms defined by the rows of $M$. For example, if we want the polynomial $p(\p{x}):= (2x_1-x_3)(x_1-3x_2+4x_3)$, then our $M$ is defined as 
\begin{eqnarray} \label{eq:hyper-14}
M:= \left[\begin{array}{ccc}
2&0&-1 \\ 1&-3&4 
\end{array} \right].
\end{eqnarray}

\subsection{Composition of hyperbolic polynomials} \label{subsec:comp}

The construction of hyperbolic polynomials using composition is based on the following result from \cite{bauschke2001hyperbolic}. 

\begin{theorem}[\cite{bauschke2001hyperbolic}-Theorem 3.1]
Suppose \(q: \R^d \rightarrow \R\) is a homogeneous symmetric polynomial of degree \(k\) on 
\(\mathbb{R}^d\), hyperbolic with respect to 
\(\p{e} := (1,1,\ldots,1)^\top \in \mathbb{R}^d\), with characteristic map \(\mu\). Let  \(p: \R^m \rightarrow \R\) be a hyperbolic polynomial of degree \(d\) that is hyperbolic with
respect to a direction \(\hat{\p{e}} \in \R^m \), with characteristic map $\lambda$.
Then the composition \(q \circ \lambda\) is a hyperbolic polynomial of degree \(n\) 
with respect to \(d\), and its characteristic map is \(\mu \circ \lambda\).
\end{theorem}

In general, calculating the characteristic map of a polynomial is not computationally efficient, however, there are cases that it can be done efficiently. For example, as discussed above, $p(\p{x}):= \ell_1(\p{x}) \cdots \ell_k(\p{x})$, where $\ell_1(\p{x}),\ldots,\ell_k(\p{x})$ are linear forms, is hyperbolic in any direction $\p{e}$ where $p(\p{e})>0$. It can be seen that
\[
\lambda_i(\p{x}) = \frac{\ell_i(\p{x})}{\ell_i(\p{e})}, \ \ \ \ i\in\{1,\ldots,k\}.
\]
\section{Domain-Driven solver and a new corrector step}
Even though DDS is a predictor-corrector interior-point code that heavily relies on the theory of duality \cite{karimi2020primal,karimi_status_arxiv}, for the problems such as HB programming and quantum relative entropy (QRE) programming \cite{karimi2025efficient} where the dual barrier and dual cone are not efficiently characterized, a primal heavy version of the algorithm is being used.  In this section, we show the details of the new corrector step that improved the performance of robustness of the primal heavy algorithm in DDS 3.0. 

A convex optimization problem $(P)$ is in the Domain-Driven form if it can be written as 
\begin{eqnarray} \label{main-p}
(P) \ \ \ \inf _{\p{x}} \{\langle \p{c},\p{x} \rangle :  A\p{x} \in D\},
\end{eqnarray}
where $\p{x} \mapsto A\p{x} : \R^n \rightarrow \R^m$ is a linear embedding, $A$ and $\p{c} \in \mathbb \R^n$ are given, and $D$ is
defined as the closure of the domain of a given $\vartheta$-\emph{self-concordant (s.c.) barrier} $\Phi$ \cite{interior-book}. The Legendre-Fenchel (LF) conjugate $\Phi_*$ of $\Phi$ is the dual barrier with the domain equal to the interior of a cone $D_*$ defined as:
 \begin{eqnarray} \label{eq:leg-conj-2} 
D_*:=\{\p{y}:  \langle \p{y} , \p{h} \rangle  \leq 0, \ \ \forall \p{h} \in \textup{rec}(D)\}, 
\end{eqnarray}
where $\textup{rec}(D)$ is the recession cone of $D$.
Let us fix an absolute constant $\xi > 1$ and define the initial points:
\begin{eqnarray} \label{starting-points-copy-2}
\p{z}^0 := \text{any vector in  $\inte D$}, \ \ \p{y}^0 :=  \Phi'(\p{z}^0), \ \ y_{\tau,0} := -\langle \p{y}^0, \p{z}^0 \rangle -\xi \vartheta.
\end{eqnarray} 
%Note that in the above definitions, we implicitly assume $z^0_\tau=1$, which is with out loss of generality by using a proper scaling for $z^0$. 
Then, it is proved in \cite{karimi2020primal} that the system 
\begin{eqnarray} \label{trans-dd-path-1-copy-2}
\begin{array}{rcl}
&(a)&  A \p{x} + \frac{1}{\tau} \p{z}^0 \in \inte D, \ \ \tau > 0,  \\
&(b)& A^\top \p{y} -A^\top \p{y}^0 = -(\tau-1) \p{c}, \ \ \p{y} \in \inte D_*,  \\
&(c)& \p{y}=\frac{\mu  }{\tau}  \Phi' \left (  A \p{x} + \frac{1}{\tau} \p{z}^0 \right), \\
%&(d)& \frac{\mu \xi}{1+\tau} = \frac{-y_{\tau,0} - \langle c, (1+\tau) x \rangle - \langle y, A x + \frac{1}{1+\tau} z^0 \rangle }{\vartheta}.
&(d)& \langle \p{c},\p{x} \rangle +  \frac{1}{\tau} \langle \p{y}, A\p{x}+\frac{1}{\tau} \p{z}^0 \rangle =  -\frac{\vartheta \xi \mu}{\tau^2} +\frac{ -y_{\tau,0}}{\tau},
\end{array}
\end{eqnarray}
has a unique solution $(\p{x}(\mu), \tau(\mu), \p{y}(\mu))$ for every $\mu > 0$. The solution set of this system for all $\mu >0$ defines our \emph{infeasible-start primal-dual central path}. Here is the definition of a corrector step:

\noindent{\bf Corrector step at a point $(\p{x},\p{y},\tau)$:} Find a point closer to the solution of \eqref{trans-dd-path-1-copy-2} for $\mu$ calculated from \eqref{trans-dd-path-1-copy-2}-(d) by substituting $(\p{x},\p{y},\tau)$. 

In DDS 2.2, the corrector step was implemented as a modified version of the main corrector step in the code and exhibited occasional numerical stability issues.  To address this, we designed a new corrector step in DDS 3.0, derived directly from the Newton system associated with the primal–dual central path equations.
%Since the corrector step in the primal-heavy version of DDS required the gradient of the dual function $\Phi_*$, it is not robust for problems such as QRE and Hyperbolic programming. Here, we come up with another corrector step that relies just on the primal s.c.\ barrier. Consider the Domain-Driven central path
%\begin{eqnarray} \label{trans-dd-path}
%\begin{array}{rcl}
%&(a)&  A x + \frac{1}{\tau} z^0 \in \inte D, \ \ \tau > 0,  \\
%&(b)& A^\top y -A^\top y^0 = -(\tau-1) c, \ \ y \in \inte D_*,  \\
%&(c)& y=\frac{\mu  }{\tau}  \Phi' \left (  A x + \frac{1}{\tau} z^0 \right), \\
%%&(d)& \frac{\mu \xi}{1+\tau} = \frac{-y_{\tau,0} - \langle c, (1+\tau) x \rangle - \langle y, A x + \frac{1}{1+\tau} z^0 \rangle }{\vartheta}.
%&(d)& \langle c,x \rangle +  \frac{1}{\tau} \langle y, Ax+\frac{1}{\tau} z^0 \rangle =  -\frac{\vartheta \xi \mu}{\tau^2} +\frac{ -y_{\tau,0}}{\tau},
%\end{array}
%\end{eqnarray}
%where $y^0$ and $z^0$ are arbitrary starting points and $\mu$ is the parameter of the path. 
For getting more efficient formulas, we define a new variable $\bar{\p{x}} := \tau \p{x}$. Note that by using some reformulations and using \eqref{trans-dd-path-1-copy-2}-(a), we can write \eqref{trans-dd-path-1-copy-2}-(d) in a linear form as:
\[
\langle \p{y},\p{z}^0 \rangle +\tau y_{\tau,0}+ \langle A^\top \p{y}^0+\p{c}, \bar{\p{x}} \rangle  = -\xi \vartheta \mu.
\]
Consider that after the predictor step, we are at a point $(\hat{\p{x}}, \hat \tau, \hat{\p{y}})$ with parameter $\mu$. The goal for the corrector step is finding a new point $(x,\tau,y)$ that satisfies 
\begin{eqnarray} \label{trans-dd-path-2}
\begin{array}{cl}

& A^\top \p{y} -A^\top \p{y}^0 +(\tau-1) \p{c}=0, \\
&\tau \p{y}-\mu   \Phi' \left (  \frac{A\bar{\p{x}}+\p{z}^0}{\tau} \right) = 0, \\
%&(d)& \frac{\mu \xi}{1+\tau} = \frac{-y_{\tau,0} - \langle c, (1+\tau) x \rangle - \langle y, A x + \frac{1}{1+\tau} z^0 \rangle }{\vartheta}.
&\langle \p{y},\p{z}^0 \rangle +\tau y_{\tau,0}+ \langle A^\top \p{y}^0+\p{c}, \bar{\p{x}} \rangle  = -\xi \vartheta \mu.
\end{array}
\end{eqnarray}
We apply the Newton method to the above system, with $(\hat{\p{x}}, \hat \tau, \hat{\p{y}})$ as the starting point. The search direction $(d\bar{\p{x}}, d\tau, d\p{y})$ must satisfy:
\begin{eqnarray} \label{trans-dd-path-3}
\begin{array}{cl}

& A^\top d\p{y} + d \tau \p{c}=0, \\
&d\tau \p{y}+\tau d\p{y} -\mu   \Phi''(\p{z})  \left (  \frac{\tau A d\bar{\p{x}}-d\tau (A\bar{\p{x}}+\p{z}^0)}{\tau^2} \right) = R, \\
%&(d)& \frac{\mu \xi}{1+\tau} = \frac{-y_{\tau,0} - \langle c, (1+\tau) x \rangle - \langle y, A x + \frac{1}{1+\tau} z^0 \rangle }{\vartheta}.
&\langle d\p{y},\p{z}^0 \rangle +d\tau y_{\tau,0}+ \langle A^\top \p{y}^0+c, d\bar{\p{x}} \rangle  = 0,
\end{array}
\end{eqnarray}
where we used $\p{z }= \frac{A\bar{\p{x}}+\p{z}^0}{\tau}$ and $R:=-\left(\hat \tau \hat{\p{y}}-\mu   \Phi' \left (  \frac{A\hat{\bar{\p{x}}}+\p{z}^0}{\hat\tau} \right)\right)$. The first and third equations are linear. We can write the second equation as
\begin{eqnarray} \label{eq:hyper-15}
d\tau \p{y}+\tau d\p{y} -\frac{\mu}{\tau} \Phi''(\p{z}) A d\bar{\p{x}} -\frac{d\tau}{\tau^2} \Phi''(\p{z}) (A\bar{\p{x}}+\p{z}^0) = R.
\end{eqnarray}

If we multiply both side of this equation from the left by $A^\top$ and use the $A^\top d\p{y} =- d \tau \p{c}$ equality, after re-ordering, we get
\[
\frac{\mu}{\tau} A^\top \Phi''(\p{z}) A d\bar{\p{x}} = d\tau \left[A^\top \p{y}-\tau\p{c} -  \frac{d\tau}{\tau^2} A^\top \Phi''(\p{z}) (A\bar{\p{x}}+\p{z}^0) \right] -A^\top R.
\]
By solving this positive definite linear system with matrix $A^\top \Phi''(\p{z}) A $, we calculate $d \bar{\p{x}}$ in terms of $d\tau$. Putting it back into  \eqref{eq:hyper-15} gives us a solution for $d\p{y}$ in terms of $d\tau$. To calculate $d\tau$, we substitute $d\bar{\p{x}}$ and $d\p{y}$ into the third equation of \eqref{trans-dd-path-3}. For the final step, to calculate $d\p{x}$ from $d\bar{\p{x}}$, note that
\[
\bar{\p{x}} = \tau \p{x}  \ \ \Rightarrow \ \ d\bar{\p{x}} = \tau d\p{x} + d\tau \p{x}.
\]
This calculated $(d\p{x},d\p{y},d\tau)$ corrector step performs more robustly for challenging HB programming and QRE programming instances in the DDS benchmark library.

\section{Numerical results for  Hyperbolic Polynomials}
The techniques designed here have been used to improve the performance of the newest version of the software package DDS (namely DDS 3.0), which can be downloaded from \cite{DDS3.0}. In addition to hyperbolic programming constraints which are the main improvements in version 3.0, DDS accepts every combination of the following function/set constraints: (1)  symmetric cones (LP, SOCP, and SDP); (2) quadratic constraints that are SOCP representable; (3) direct sums of an
arbitrary collection of 2-dimensional convex sets defined as the epigraphs of univariate convex
functions (including as special cases geometric programming and entropy programming); (4) generalized  Koecher (power) cone; (5) epigraphs of matrix norms (including
as a special case minimization of nuclear norm over a linear subspace); (6) vector relative entropy; and (7) epigraphs of quantum
entropy and quantum relative entropy. In this section, we present several numerical examples of running DDS 3.0 for HB programming. We performed computational experiments using the software MATLAB R2024a, on a 1.7 GHz 12th Gen Intel Core i7 personal computer with 32GB of memory. All the numerical results in this section are by using the default settings of DDS, including the tolerance of $tol=10^{-8}$.

The examples created and tested here are included in  a library of hyperbolic cone programming problems \cite{Karimi_Library_for_Modern}:\\
\href{https://github.com/mehdi-karimi-math/Library-for-Modern-Convex-Optimization/tree/main/Hyperbolic\%20Programming}{https://github.com/mehdi-karimi-math/Library-for-Modern-Convex-Optimization}

To create the tables of this section, first install DDS 3.0 \cite{DDS3.0} and then follow the instructions in \cite{Karimi_Library_for_Modern}. In general, to add a hyperbolic programming constraint to DDS, consider a constraint of the form
\begin{eqnarray} \label{eq:hyper-4}
% p(Ax+b) \leq 0, \ \  i \in \{1,\ldots,\ell\}. 
p(A\p{x}+\p{b}) \geq 0.
\end{eqnarray}
To input this constraint to DDS as the $k$th block, the variable \tx{cons} which contains the structure of each constraint is defined as follows:\\

\noindent \tx{cons\{k,1\}='HB'},\\
\noindent \tx{cons\{k,2\}=} number of variables in $p(\p{z})$. \\
\noindent \tx{cons\{k,3\}} is the \tx{poly} that can be given in one of the three formats of Section \ref{sec:poly-format}. \\
\noindent \tx{cons\{k,4\}} is the format of polynomial that can be $\tx{'monomial'}$, $\tx{'straight\_line'}$, or $\tx{'determinant'}$. \\
\noindent \tx{cons\{k,5\}} is the direction of hyperbolicity or a vector in the interior of the hyperbolicity cone. 

\begin{example}
Assume that we want to give constraint \eqref{eq:hyper-4} to DDS for $p(\p{x}):=x_1^2-x_2^2-x_3^2$, using the monomial format. Then, \tx{cons} part is defined as
\begin{eqnarray*} \label{eq:hyper-5}
\tx{cons\{k,1\}='HB'}, \ \ \tx{cons\{k,2\}}=[3], 
\tx{cons\{k,3\}}=\left[\begin{array}{cccc} 2&0&0&1 \\ 0&2&0&-1 \\ 0&0&2&-1 
\end{array}	\right], \\
\tx{cons\{k,4\}='monomial'}, \ \ \tx{cons\{k,5\}}=\left[\begin{array}{ccc} 1 & 0 &0\end{array}\right]^ \top.
\end{eqnarray*}
\end{example}

For a problem in the Domain-Driven form \eqref{main-p}, for every point $ \p{x} \in \R^n$ such that $A \p{x} \in D$ and every point $ \p{y} \in D_*$ such that $A^\top \p{y} = -\p{c}$, \emph{the duality gap} is defined as:
\begin{eqnarray} \label{eq:duality-gap-1}
\langle \p{c}, \p{x} \rangle + \delta_*(\p{y}|D),
\end{eqnarray}
where
\begin{eqnarray}  \label{eq:supp-fun-1}
\delta_*(\p{y}|D) := \sup\{ \langle \p{y},\p{z} \rangle :  \p{z} \in D\},  \ \ \ \text{(support function of $D$).}
\end{eqnarray}
Lemma 2.1 in \cite{karimi2020primal} shows that duality gap is well-defined and zero duality gap is a guarantee for optimality.
% Consider a conic optimization problem
%\begin{eqnarray}  \label{eq:duality-1}
%\begin{array}{rrl}
%\min  &   \langle c , x \rangle  &   \\
%\text{s.t.}  &   Ax+b    & \in K,  \\
%\end{array}
%\end{eqnarray}
%where $K \subset \R^m$ is a closed convex cone and $A$, $c$, and $b$ are matrices and vectors of proper size.  
For the special case of conic optimization, where $D=K-\p{b}$ for a closed  convex cone $K$ and a given vector $\p{b}$, we can show that the Domain-Driven duality gap is reduced to the well-known 
\[
\text{Duality Gap} =  \langle \p{c} , \p{x} \rangle +  \langle \p{b} , \p{y} \rangle.
\]
Note that for the hyperbolicty cones, we do not have an efficient oracle to check membership in the dual cone. However, the path-following algorithm in DDS and the properties of the self-concordant functions guarantee that the dual iterate $y$ lies in the dual cone. The default $tol$ in DDS is $10^{-8}$, which means that the algorithm stops when the duality gap, the primal feasibility, and the dual feasibility, are below this tolerance. In other words:
\[
 \text{\bf{DDS Stopping Criteria}}:  \max\{\text{primal feasibility, dual feasibility, duality gap}\}  \leq tol.
\]
 
\subsection{Minimizing entropy combined with HP programming}
In this section, we investigate an optimization problem of minimizing the entropy of a vector under both hyperbolic and linear inequality constraints. This benchmark highlights the efficiency and robustness of the proposed derivative computations and provides a direct comparison between DDS 3.0 and the previous version DDS 2.2, which uses our old matrix version of a SLP. The problem is formulated as follows:
\begin{eqnarray}  \label{eq:hyper-16}
\begin{array}{rrl}
\min  &   \text{entr}(\p{x})   &   \\
\text{s.t.}  &   A\p{x}+\p{b}    & \in \Lambda_+(p,\p{e})    \\
      &   \p{x}   & \geq \gamma \mathbf{1},
\end{array}
\end{eqnarray}
where $\mathbf{1}$ is the vector of all ones. In this example, $p(\p{x}) := e_k^n(\p{x})$ is the elementary symmetric polynomial. The matrix $A \in \R^{m \times n}$ is created randomly with $\pm 1$ entries and $\p{b}$ is the vector of all ones. $p(\p{x})$ is represented using the SLP as discussed in Section \ref{SLP}. The results are shown in Table \ref{table:HB3}. As can be seen, DDS 3.0 can handle HB polynomials with as many as $2.4e57$ monomials. The performance of DDS 2.2 is still acceptable for small size problems, but the running time grows much faster than DDS 3.0 as $n$ increases. 

\begin{table} [!ht] 
  \caption{Results for problems formulated as \eqref{eq:hyper-16}. Times are in seconds. $A$ is created randomly with $\{0,1\}$ entries and $b$ is the vector of all ones. For DDS 2.2, the old format of a SLP was used. }
  \label{table:HB3}
  \renewcommand*{\arraystretch}{1.3}
  \begin{tabular}{ |c | c| c | c | c |  c| c|c|}
    \hline
\multirow{2}{*}{$(n,k)$}  &    \multirow{2}{*}{size of $x$}   &  \multirow{2}{*}{\# of monomials}   &  \multirow{2}{*}{$\gamma$} & \multicolumn{2}{|c|}{DDS 2.2}  & \multicolumn{2}{|c|}{DDS 3.0} \\    \cline{5-8}  
&&&&  Iter & Time&   Iter & Time\\ \hline 
(20,5)  & 10  & 1.6e+04 & 0.5   & 13  & 9.97 &  12  &  0.84   \\ \hline
(20,10)  & 10  & 1.8e+05 & 0.5   & 14  & 24.44 &  13  &  1.41   \\ \hline
(20,15)  & 10  & 1.6e+04 & 0.5   & 15 & 41.76 &  14  &  2.40   \\ \hline
(50,10)  & 10  & 1.0e+10 & 0.5   & 14 & 459&  13  &  5.09   \\ \hline
(50,25)  & 10  & 1.3e+14 & 0.5   & 16  & 4224 &  14  &  9.63   \\ \hline
(50,30)  & 10  & 4.7e+13 & 0.5   &  \multicolumn{2}{|c|}{time$>5000$} &  15  &  11.18   \\ \hline
(100,10)  & 10  & 1.7e+13 & 0.5   &  \multicolumn{2}{|c|}{time$>5000$} &  13  &  9.88   \\ \hline
(100,20)  & 10  & 5.4e+20 & 0.5   &  \multicolumn{2}{|c|}{time$>5000$} &  14  &  20.20   \\ \hline
(100,30)  & 10  & 2.9e+25 & 0.5   &  \multicolumn{2}{|c|}{time$>5000$} &  15  &  30.74   \\ \hline
(200,10)  & 10  & 2.2e+16 & 0.5   &  \multicolumn{2}{|c|}{time$>5000$} &  13  &  25.80   \\ \hline
(200,20)  & 10  & 1.6e+27 & 0.5   & \multicolumn{2}{|c|}{time$>5000$} &  14  &  50.73   \\ \hline
(200,30)  & 10  & 4.1e+35 & 0.5   &  \multicolumn{2}{|c|}{time$>5000$} &  15  &  78.88   \\ \hline
(500,10)  & 10  & 2.5e+20 & 0.5   &  \multicolumn{2}{|c|}{time$>5000$} &  13  &  111.91   \\ \hline
(500,20)  & 10  & 2.7e+35 & 0.5   &  \multicolumn{2}{|c|}{time$>5000$} &  14  &  234.39   \\ \hline
(500,30)  & 10  & 1.4e+48 & 0.5   &  \multicolumn{2}{|c|}{time$>5000$} &  15  &  382.88   \\ \hline
(1000,10)  & 10  & 2.6e+23 & 0.5   &  \multicolumn{2}{|c|}{time$>5000$} &  13  &  549.40   \\ \hline
(1000,20)  & 10  & 3.4e+41 & 0.5   &  \multicolumn{2}{|c|}{time$>5000$} &  14  &  1172.23   \\ \hline
(1000,30)  & 10  & 2.4e+57 & 0.5   &  \multicolumn{2}{|c|}{time$>5000$} &  15  &  2053.65   \\ \hline
(2000,10)  & 10  & 2.8e+26 & 0.5   &  \multicolumn{2}{|c|}{time$>5000$} &  13  &  3513.78   \\ \hline

%(20,5)  & $20*100$  &   15504 &  0.5 & 22/\ 14.34 & 21/\ 3    \\ \hline
%(20,10)  & $20*100$  &   184756 &  0.5 & 25/\ 39 & 23/\ 4.5    \\ \hline
%(50,10)  & $50*100$  &   $\geq$ 1e10 &  0.5 & 26/\ 470 & 24/\ 12.2    \\ \hline
%(50,20)  & $50*100$  &   $\geq$ 4.7e13 &  0.5 & 29/\ 2158 & 26/\ 25    \\ \hline
%(100,10)  & $100*100$  &   $\geq$ 1.7e13 &  0.5 &  time $>$ 3000 & 23/\ 33.4  \\ \hline
%(100,20)  & $100*100$  &   $\geq$ 5.3e20 &  0.5 & time $>$ 3000 & 26/\ 60.2     \\ \hline
\end{tabular}
\end{table}

\subsection{Projection onto the hyperbolicity cone}
These problems are studied in \cite{nagano2024projection} using a first-order dual Frank-Wolfe (FW) approach. The problem can be formulated as
\begin{eqnarray}  \label{eq:hyper-17}
\begin{array}{rrl}
\min  &   t   &   \\
\text{s.t.}  &   \|\p{x}-\p{c}\|_2  & \leq t     \\
      &   \p{x}   & \in \Lambda_+(p,\p{e}),    
\end{array}
\end{eqnarray}
where $c$ is an arbitrary vector. The projection problem under consideration involves projecting a vector $c$ onto the hyperbolicity cone defined by the 2-norm. The numerical experiments of \cite{nagano2024projection} used elementary symmetric polynomials to model $p(\p{x})$. Benchmarking against DDS 2.2, they reported that ``for $(n, k) = (20, 10)$ and $(n, k) = (30, 15)$, DDS struggles to complete a single iteration.'' In contrast, by representing elementary symmetric polynomials through straight-line programs (SLPs), we have substantially improved the performance of DDS, enabling it to solve instances where the earlier implementation failed. For the FW, the divide-and-conquer version of the code designed for ESP was used. To define the FW gap used in our experiments, consider the optimization problem $\min_{\p{x} \in C} f(\p{x})$ for $C$ being a convex and compact set in $\R^n$ and $f$ being a convex function. Let $\p{x}_k$ be the iterates in the FW algorithm, and for each $\p{x}_k$, let us define $\p{s}_k \in \text{argmin}_{\p{x} \in C} \langle \nabla f(\p{x}_k) ,\p{x} \rangle$. The FW gap $G(x)$ can be defined as
\[
G(\p{x}) := \max_{\p{s} \in C} \langle -\nabla f(\p{\p{x}}), \p{s}-\p{x} \rangle.
\]
We can check that $G(\p{x})$ is non-negative for the iterates and is an upper bound for the optimality gap.

For the FW algorithm \cite{nagano2024projection} , in addition to $tol$ which is the bound on the FW gap, we can base the stopping criteria on the number of iterations and the running time. For our experiments, we set the limit on running time equal to $10^3$ seconds. For each value of $(n,k)$, 10 random vectors are chosen as $\p{c}$ in \eqref{eq:hyper-17} based on the normal distribution using the \texttt{randn} command in MATLAB, with average 0 and std $0.5$. The results are summarized in Table~\ref{table:HB4}. Following the presentation style of \cite{nagano2024projection}, we also report the infinity-norm distance between the solutions returned by FW and DDS. For FW, we used the option ``symmetric" for the polynomial type and we tested two tolerance levels: the default $\texttt{tol}=10^{-2}$, which converges rapidly but with limited accuracy, and the tighter $\texttt{tol}=10^{-4}$, which achieves higher precision at the cost of additional runtime. Based on our experiments, for a fixed value of \( n \), the FW algorithm performs efficiently and yields accurate results for small values of \( k \). However, its performance deteriorates as \( k \) increases. For \( n > 30 \), the FW algorithm frequently has numerical failure when \( k \geq 30 \) because of creating NaN values for the input of the MATLAB's \texttt{roots} function for calculating eigenvalues. Another notable observation is that the FW algorithm is more sensitive to the choice of the input vector \( \p{c} \). This sensitivity is reflected in the higher relative standard deviation of the running time compared to DDS, indicating less consistency across problem instances.

\begin{table} [!ht] 
  \caption{Results for the projection problems formulated as \eqref{eq:hyper-17} where $p(\p{x})$ is an elementary symmetric polynomial, comparing DDS 3.0 and the FW algorithm in \cite{nagano2024projection}. The numerical failures happen for the FW algorithm when MATLAB stops running because of creating NaN values for the input of the MATLAB's \texttt{roots} function for calculating eigenvalues. Times are in seconds. }
  \label{table:HB4}
  \renewcommand*{\arraystretch}{1.3}
  \small
  \begin{tabular}{ |l | c| c |c|c| c | c |  c| c|}
    \hline
\multirow{2}{*}{$(n,k)$}   & \multirow{2}{*}{ \# of }     &  \multicolumn{2}{|c|}{FW, tol = $1e-2$}& \multicolumn{2}{|c|}{FW, tol = $1e-4$}& \multicolumn{2}{|c|}{DDS 3.0} \\  \cline{3-8}     & mono &  FW time & $\|x-x_{DDS}\|_\infty$ &  FW time & $\|x-x_{DDS}\|_\infty$  &  Iter & time \\
\hline 
(20,5)  &   1.6e04   & 0.05$\pm$0.12 & 5.68e-02$\pm$2.54e-02  &  4.84$\pm$3.58 & 8.19e-04$\pm$6.50e-04  &  8.7$\pm$0.7& 1.22$\pm$0.70   \\ \hline
(20,10)  &   1.8e05   & 0.04$\pm$0.02 & 2.22e-02$\pm$1.37e-02  &  7.80$\pm$2.49 & 1.92e-04$\pm$5.84e-05  &  12.4$\pm$1.8& 1.90$\pm$1.01   \\ \hline
(20,15)  &   1.6e04   & 0.02$\pm$0.01 & 3.00e-02$\pm$1.20e-02  &  17.58$\pm$11.10 & 5.59e-04$\pm$2.27e-04  &  17.1$\pm$3.7& 2.66$\pm$1.50   \\ \hline
(50,10)  &   1.0e10   & 0.06$\pm$0.04 & 1.20e-01$\pm$2.69e-02  &  31.08$\pm$13.35 & 3.88e-04$\pm$2.19e-04  &  10.4$\pm$1.2& 3.45$\pm$0.26   \\ \hline
(50,25)  &   1.3e14   & 1.10$\pm$0.80 & 1.69e-02$\pm$2.59e-03  &  52.52$\pm$12.77 & 1.80e-04$\pm$2.79e-05  &  17.0$\pm$1.9& 10.95$\pm$1.05   \\ \hline
(50,30)  &   4.7e13   & 1000.17$\pm$0.11 & 5.26e-01$\pm$7.37e-02  &  1000.11$\pm$0.10 & 5.25e-01$\pm$7.36e-02  &  20.1$\pm$2.5& 14.77$\pm$1.92   \\ \hline
(50,40)  &   1.0e10    & \multicolumn{2}{|c|}{Numerical Failure}   &  \multicolumn{2}{|c|}{Numerical Failure}  &  22.0$\pm$2.2 & 19.95$\pm$2.11   \\ \hline
(100,10)  &   1.7e13   & 0.11$\pm$0.01 & 7.39e-02$\pm$3.46e-02  &  78.29$\pm$43.69 & 4.41e-03$\pm$3.43e-03  &  10.1$\pm$0.3& 10.34$\pm$6.00   \\ \hline
(100,20)  &   5.4e20   & 0.43$\pm$0.13 & 1.25e-01$\pm$3.94e-02  &  160.88$\pm$35.03 & 1.91e-04$\pm$1.23e-05  &  12.2$\pm$0.6& 18.39$\pm$1.23   \\ \hline
(100,30)  &   2.9e25    & \multicolumn{2}{|c|}{Numerical Failure}   &  \multicolumn{2}{|c|}{Numerical Failure}  &  15.3$\pm$0.8& 31.67$\pm$2.09   \\ \hline
(100,40)  &   1.4e28    & \multicolumn{2}{|c|}{Numerical Failure}   &  \multicolumn{2}{|c|}{Numerical Failure}  &  17.6$\pm$1.6 & 54.00$\pm$22.08   \\ \hline
(200,10)  &   2.2e16   & 0.56$\pm$0.01 & 2.11e-02$\pm$1.44e-02  &  41.38$\pm$60.64 & 1.61e-02$\pm$9.38e-03  &  10.5$\pm$0.5& 22.08$\pm$2.64   \\ \hline
(200,20)  &   1.6e27   & 0.65$\pm$0.21 & 2.00e-01$\pm$1.12e-01  &  519.47$\pm$174.12 & 1.15e-03$\pm$1.02e-03  &  11.7$\pm$1.4& 47.73$\pm$8.05   \\ \hline
(200,30)  &   4.1e35    & \multicolumn{2}{|c|}{Numerical Failure}   &  \multicolumn{2}{|c|}{Numerical Failure}  &  14.3$\pm$1.3& 113.81$\pm$63.24   \\ \hline
(200,40)  &   2.1e42    & \multicolumn{2}{|c|}{Numerical Failure}   &  \multicolumn{2}{|c|}{Numerical Failure}  &  15.3$\pm$1.2 & 159.43$\pm$133.19   \\ \hline
(500,10)  &   2.5e20   & 6.35$\pm$0.10 & 3.45e-03$\pm$2.60e-03  &  6.32$\pm$0.07 & 3.45e-03$\pm$2.60e-03  &  12.1$\pm$0.3& 104.29$\pm$4.11   \\ \hline
(500,20)  &   2.7e35   & 6.33$\pm$0.05 & 3.23e-02$\pm$1.28e-02  &  223.49$\pm$296.03 & 2.75e-02$\pm$5.67e-03  &  12.3$\pm$1.6& 338.69$\pm$276.17   \\ \hline
(500,30)  &   1.4e48    & \multicolumn{2}{|c|}{Numerical Failure}   &  \multicolumn{2}{|c|}{Numerical Failure}  &  13.4$\pm$1.5& 394.89$\pm$83.14   \\ \hline
(500,40)  &   2.2e59    & \multicolumn{2}{|c|}{Numerical Failure}   &  \multicolumn{2}{|c|}{Numerical Failure}  &  15.2$\pm$0.8 & 598.42$\pm$51.70   \\ \hline
(1000,10)  &   2.6e23   & 41.51$\pm$6.92 & 4.70e-04$\pm$2.38e-04  &  41.50$\pm$6.88 & 4.70e-04$\pm$2.38e-04  &  12.7$\pm$0.7& 533.23$\pm$26.94   \\ \hline
(1000,20)  &   3.4e41   & 43.67$\pm$0.25 & 6.36e-03$\pm$2.37e-03  &  43.61$\pm$0.13 & 6.36e-03$\pm$2.37e-03  &  12.7$\pm$0.5& 1110.61$\pm$40.03   \\ \hline
(1000,30)  &   2.4e57    & \multicolumn{2}{|c|}{Numerical Failure}   &  \multicolumn{2}{|c|}{Numerical Failure}  &  12.8$\pm$0.6& 2330.12$\pm$1218.36   \\ \hline
(1000,40)  &   5.6e71    & \multicolumn{2}{|c|}{Numerical Failure}   &  \multicolumn{2}{|c|}{Numerical Failure}  &  13.1$\pm$0.6 & 2991.57$\pm$292.00   \\ \hline

\end{tabular}
\end{table}
We have solved the projection problem for other hyperbolic polynomials as well. Since these polynomials are not coded for the FW algorithm in \cite{nagano2024projection}, we only report the results for DDS 3.0. Table \ref{table:HB5} shows the results of projecting a random vector into the cone created by Vamos-like polynomials. The running time averages and stds are shown for ten difference vectors created as above.  As can be seen, DDS 3.0 can solve the projection problem for Vamos-like polynomials with 3000 variables under $5000$ seconds. Table \ref{table:HB6} solves projection problem for HB polynomials created by composition, as explained in Subsection \ref{subsec:comp}. 

\begin{table} [!ht] 
  \caption{Results for problems formulated as \eqref{eq:hyper-17} where $p(\p{x})$ is an Vamos-like polynomial defined in \eqref{eq:hyper-9}. Times are in seconds. }
  \label{table:HB5}
  \renewcommand*{\arraystretch}{1.3}
  \begin{tabular}{ |c | c| c | c | c |  }
    \hline
$m$  &    \# of variables  &  \# of monomials   &  Iter  & Time \\ \hline 
10  & 20  &  4.8e3    &  8.6$\pm$0.5 & 0.70$\pm$0.07   \\ \hline
  25  & 50  &  2.3e5    &  8.9$\pm$0.3 & 1.79$\pm$0.20   \\ \hline
  50  & 100  &  3.9e6    &  10.3$\pm$0.5 & 4.38$\pm$0.38   \\ \hline
  100  & 200  &  6.4e7    &  11.1$\pm$0.7 &12.62$\pm$0.72   \\ \hline
  250  & 500  &  2.5e9    &  12.8$\pm$0.4 & 61.42$\pm$3.98   \\ \hline
  500  & 1000  &  4.1e10    &  14.0$\pm$0.0 & 284.12$\pm$31.02   \\ \hline
1000  & 2000  &  6.6e11    &  15.4$\pm$0.8 & 1669.93$\pm$70.88   \\ \hline
1500  & 3000  &  3.3e12    &  15.9$\pm$0.9 & 4874.51$\pm$263.40   \\ \hline

%5  & $10$  &   202 &  9/\ 0.5    \\ \hline
%10  & $20$  &   4827 &   10/\ 1   \\ \hline
%20  & $40$  &   91352 &   10/\ 1.5    \\ \hline
%50  & $100$  &   $\geq$ 3.9e6 &   11/\ 4.9    \\ \hline
%100  & $200$  &   $\geq$ 6.4e7 &   11/\ 12.1  \\ \hline
%500  & $1000$  &   $\geq$ 4.14e10 &   14/\ 298  \\ \hline
%1000  & $2000$  &   $\geq$ 6.6e11 &   14/\ 1442  \\ \hline
\end{tabular}
\end{table}

\begin{table} [!ht] 
  \caption{Results for problems formulated as \eqref{eq:hyper-17} where $p(\p{x})$ is a ESP composed with the characteristic map of a product of linear forms. $(n,k)$ are the parameters of the ESP, where $m$ is the number of variables. Times are in seconds. }
  \label{table:HB6}
  \renewcommand*{\arraystretch}{1.3}
  \begin{tabular}{ |c | c| c | c | c |  c| c|}
    \hline
  ESP   &   \multirow{2}{*}{\# of monomials}  & Linear form  &    \multicolumn{2}{|c|}{DDS 3.0} \\   \cline{4-5}
$(n,k)$  &     &     (terms,var)   & Ite & time \\ \hline 
(20,10)  &   1.8e+05    &    (20,5)  &  7 & 1.51  \\ \hline
(50,10)  &   1.0e+10    &    (50,5)  &  8 & 4.48  \\ \hline
(50,25)  &   1.3e+14    &    (50,5)  &  11 & 10.57  \\ \hline
(50,30)  &   4.7e+13    &    (50,10)  &  7 & 9.81  \\ \hline
(100,10)  &   1.7e+13    &    (100,10)  &  5 & 7.90  \\ \hline
(100,20)  &   5.4e+20    &    (100,10)  &  12 & 21.83  \\ \hline
(100,30)  &   2.9e+25    &    (100,10)  &  7 & 21.11  \\ \hline
(200,10)  &   2.2e+16    &    (200,10)  &  7 & 20.77  \\ \hline
(200,20)  &   1.6e+27    &    (200,10)  &  8 & 38.05  \\ \hline
(200,30)  &   4.1e+35    &    (200,10)  &  7 & 41.60  \\ \hline
(500,10)  &   2.5e+20    &    (500,10)  &  9 & 54.30  \\ \hline
(500,20)  &   2.7e+35    &    (500,10)  &  9 & 103.52  \\ \hline
(500,30)  &   1.4e+48    &    (500,20)  &  9 & 149.99  \\ \hline
(1000,10)  &   2.6e+23    &    (1000,10)  &  9 & 109.68  \\ \hline
(1000,20)  &   3.4e+41    &    (1000,10)  &  9 & 215.20  \\ \hline
(2000,10)  &   2.8e+26    &    (2000,10)  &  8 & 203.82  \\ \hline
(2000,20)  &   3.9e+47    &    (2000,10)  &  8 & 416.09  \\ \hline
(3000,10)  &   1.6e+28    &    (3000,10)  &  7 & 289.81  \\ \hline
(3000,20)  &   1.3e+51    &    (3000,10)  &  8 & 586.23  \\ \hline
(4000,10)  &   2.9e+29    &    (4000,10)  &  6 & 345.90  \\ \hline
(4000,20)  &   4.3e+53    &    (4000,10)  &  9 & 866.25  \\ \hline

\end{tabular}
\end{table}

\subsection{DDS versus SDP relaxations and homotopy method} \label{sec:SDP}
One of the most popular theoretical question about hyperbolic polynomials is the \emph{Generalized Lax Conjecture} which states that all hyperbolicity cones are spectrahedral \cite{helton2007linear,vinnikov2012lmi}. Mathematically, for every hyperbolic polynomial $p$ in the direction of $\p{e}$, there exist symmetric matrices $A_i$ such that 
\begin{eqnarray} \label{eq:Lax}
\Lambda_+(p,\p{e}) = \left\{\p{x} \in \R^m :  \sum_{i=1}^m x_i A_i \ \text{is positive semidefinite} \right\}.
\end{eqnarray}
Some special cases of the theorem has been proved \cite{branden2014hyperbolicity,amini2019}, however, the general problem remains wide open, with deep connections to convex optimization, real algebraic geometry, and combinatorics. 
At a first glance, this may suggest that any hyperbolic programming (HP) problem can be reformulated as a semidefinite program (SDP) and solved using existing efficient SDP solvers.
However, the generalized Lax conjecture does not provide any bound on the dimension of the matrix pencil, and in practice, this dimension can grow extremely fast.
Consequently, even though an SDP representation may exist theoretically, it can be computationally intractable, with matrix sizes far exceeding the number of decision variables. There are explicit theoretical results (\cite{RRSW2019,Oliveira2020}
indicating that for many HB programs, their most efficient SDP representation will be intractable.
It was also observed in \cite{DDS} that for homogeneous cones (which are a special subclass of HB cones) there are significant computational advantages in utilizing the algebraic structure of homogeneous cones while treating
nuclear-norm minimization problems rather than solely relying
on SDP formulations of nuclear-norm minimization problems (the latter less efficient approach is equivalent
to using an SDP representation of the underlying homogeneous cone).

This limitation can already be observed in the special case of hyperbolicity cones defined by elementary symmetric polynomials, for which spectrahedral representations are known \cite{branden2014hyperbolicity, saunderson2015polynomial}. In both formulations, the dimension of the lifted SDP grows in a way that even moderately sized instances are computationally infeasible for current SDP solvers.
A recent study \cite{klingler2025homotopy} introduced a homotopy-based method for solving general convex optimization problems, which can also be applied to HP problems. The authors compared their approach against the SDP relaxation derived from \cite{branden2014hyperbolicity} and analyzed the scalability of the corresponding matrix size. They showed that the dimension 
$s(n,k)$ of the matrix pencil satisfies
\begin{eqnarray*}
s(n,k) = \sum_{i=0}^{k-1}{n+1 \choose i}. i!  \geq \left(\frac{n+1}{k-1}\right)^{k-1}.(k-1)!, 
\end{eqnarray*}
where they used the lower bound ${n \choose k} \geq (n/i)^i$, which grows very fast. 

To compare DDS 3.0 with the homotopy method and SDP relaxation, we solve the exact problem given in  \cite{klingler2025homotopy} with DDS 3.0 and compare the results. The problem is 
\[
\begin{array}{rrl}
\max  &   f(\p{x})   &   \\
\text{s.t.}  &   \p{x} \in &  \Lambda(p,\p{1})\cap \left[\p{1}+\left\{\p{x} \in \R^n: \p{1}^\top \p{x} = 0 \right\}\right]   \\
\end{array}
\]
where $p$ is an ESP with parameter $(n,k)$ and $f(\p{x})$ is a linear function. 

The homotopy code is not publicly available, so we use the exact numerical values reported in Figure 5 of \cite{klingler2025homotopy}. In Figure~\ref{fig:homo}, we compare the performance of DDS with the homotopy method and the SDP relaxation approach. The results for DDS are averaged over 10 random creations of the objective function. It is evident from the figure that the growth rate of the running times differs significantly across the methods. While both the homotopy method and the SDP relaxation exhibit rapidly increasing computational costs, becoming intractable even for moderate problem sizes, the running time of DDS grows almost linearly on a logarithmic scale, showing its superior scalability.

\begin{figure} [h!]
\centering
\begin{subfigure}{.5\textwidth}
  \centering
  \includegraphics[width=\linewidth]{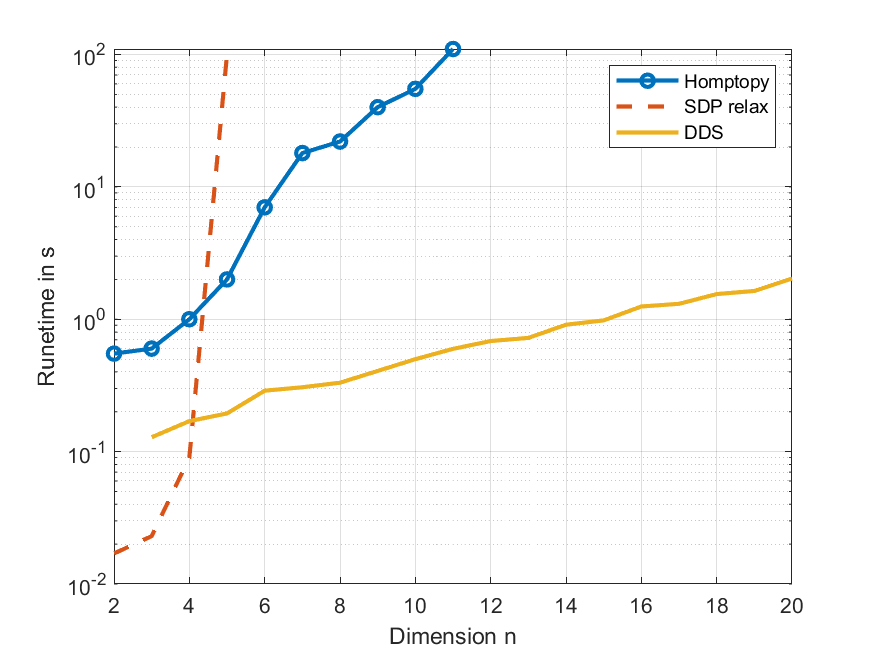}
  \caption{$k=n-1$}
  \label{fig:sub1}
\end{subfigure}%
\begin{subfigure}{.5\textwidth}
  \centering
  \includegraphics[width=\linewidth]{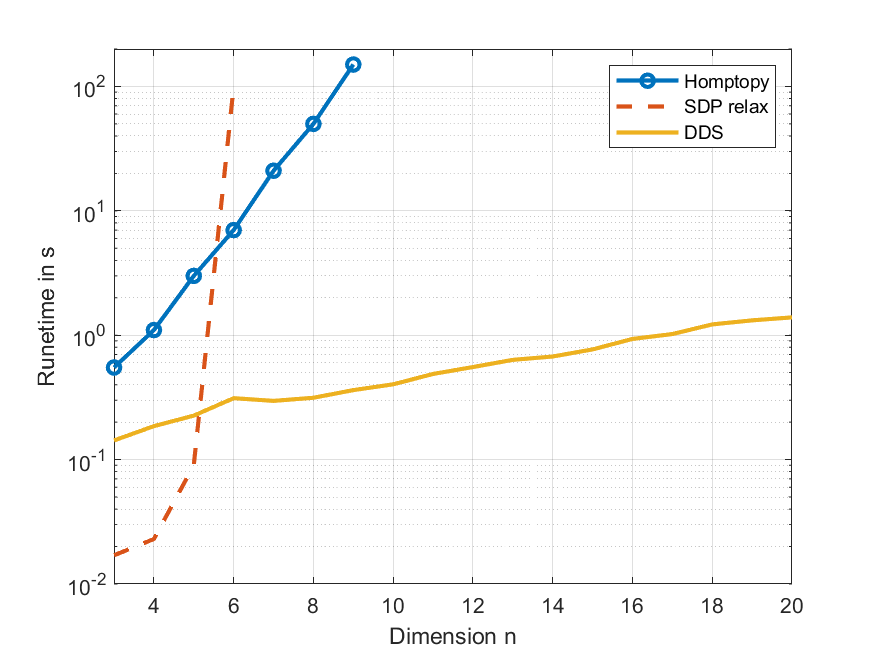}
  \caption{$k=\lfloor \frac{n+1}{2} \rfloor$}
  \label{fig:sub2}
\end{subfigure}
\caption{Comparing DDS with the results presented in \cite{klingler2025homotopy}-Figure 5 for the homotopy method and the SDP relaxation. As can be seen, DDS clearly outperforms both methods. }
\label{fig:homo}
\end{figure}

\subsection{Hyperbolic and Quantum Relative Entropy Programming}
To consider the combination of quantum relative entropy (QRE) and HB constraints, the two classes that are heavily improved in the last two version of DDS, we consider the following problem
\begin{eqnarray} \label{eq:QRE-HB-1}
\begin{array}{ccll}
\textup{min} &  & qre\left(I+\sum_{i=1}^n x_i A_i, I+\sum_{i=1}^n y_i B_i\right) & \\
 & & x  \ \ \in \Lambda_+(p,e) , &  \\
 & & x \ \ \geq \gamma \mathbbm{1}, &  \\
\end{array}
\end{eqnarray}
where $A_i \in \S^m$ and $B_i \in \S^m$, $i\in\{1,\ldots,m\}$, are sparse 0,1 random matrices (each matrix is all zero except for two off-diagonal entries). Table \ref{table:HB7} shows the results of solving this problem with DDS 3.0.  For each pair $(m,n)$, we consider two problem instances: one with $\gamma=0.1$, for which the problem admits an optimal solution, and one with $\gamma = 1.0$, for which the problem is infeasible.

\begin{table} [!ht] 
  \caption{Results for problems formulated as \eqref{eq:QRE-HB-1} where $p(\p{x})$ is a ESP with parameter $(n,k)$ and $m$ is the dimension of the matrices in the matrix pencil. Times are in seconds. }
  \label{table:HB7}
  \renewcommand*{\arraystretch}{1.3}
  \begin{tabular}{ |c | c| c | c | c |  c|c| }
    \hline
$m$  & $n$ &  $k$  &  $\gamma$ & Status &  Iter  & time \\ \hline 
10   &  10   &   5  &   0.1 & Opt Soln & 11   &   0.6   \\ \hline
10   &  10   &   5  &   1.0 & Infeasible & 16   &   0.7   \\ \hline
20   &  10   &   5  &   0.1 & Opt Soln & 12   &   1.8   \\ \hline
20   &  10   &   5  &   1.0 & Infeasible & 19   &   1.1   \\ \hline
30   &  50   &   10  &   0.1 & Opt Soln & 18   &   8.2   \\ \hline
30   &  50   &   10  &   1.0 & Infeasible & 22   &   10.1   \\ \hline
30   &  100   &   10  &   0.1 & Opt Soln & 22   &   20.6   \\ \hline
30   &  100   &   10  &   1.0 & Infeasible & 24   &   19.9   \\ \hline
50   &  50   &   10  &   0.1 & Opt Soln & 14   &   14.8   \\ \hline
50   &  50   &   10  &   1.0 & Infeasible & 24   &   13.3   \\ \hline
50   &  200   &   10  &   0.1 & Opt Soln & 31   &   71.0   \\ \hline
50   &  200   &   10  &   1.0 & Infeasible & 28   &   59.1   \\ \hline
100   &  50   &   10  &   0.1 & Opt Soln & 14   &   21.0   \\ \hline
100   &  50   &   10  &   1.0 & Infeasible & 29   &   35.4   \\ \hline
100   &  500   &   10  &   0.1 & Opt Soln & 33   &   529.3   \\ \hline
100   &  500   &   10  &   1.0 & Infeasible & 36   &   355.4   \\ \hline
100   &  1000   &   10  &   0.1 & Opt Soln & 51   &   2710.4   \\ \hline
100   &  1000   &   10  &   1.0 & Infeasible & 39   &   1759.0   \\ \hline

\end{tabular}
\end{table}

\subsection{Unbounded Hyperbolic Programming Instances}
In order to provide a comprehensive set of benchmark instances, we introduce a class of unbounded hyperbolic programming problems. These problems involve the direct sum of two hyperbolicity cones generated by elementary symmetric polynomials (ESP) as defined in \eqref{eq:hyper-unb}, where $p_1$ and $p_2$ are ESPs with parameters $(n,k_1)$ and $(n,k_2)$, respectively. This setup emphasizes that in DDS, we can define a set as a direct sum of an arbitrary number of hyperbolicity cones and other covered convex sets.  The results of solving the problem with DDS 3.0 are given in Table \ref{table:HB-unb}. When the objective function contains $\p{x}-\p{y}$, the problem is unbounded, and when it contains $\p{x}+\p{y}$, the problem has an optimal solution. 

\begin{eqnarray}  \label{eq:hyper-unb}
\begin{array}{rrl}
\min  & \sum_{i=1}^n  (\p{x} \pm \p{y})_i   &   \\
\text{s.t.}  &   \p{x}   & \in \Lambda_+(p_1,\p{e_1})    \\
  &   \p{y}   & \in \Lambda_+(p_2,\p{e_2})  
\end{array}
\end{eqnarray}

\begin{table} [!ht] 
  \caption{ \small Results for problems formulated as \eqref{eq:hyper-unb} where $p_1$ and $p_2$ are ESPs with parameters $(n,k_1)$ and $(n,k_2)$, respectively. When the objective function contains $\p{x}-\p{y}$, the problem is unbounded, and when it contains $\p{x}+\p{y}$, the problem has an optimal solution. Times are in seconds. }
  \label{table:HB-unb}
  \renewcommand*{\arraystretch}{1.3}
  \begin{tabular}{ |c | c| c | c | c | c|  }
    \hline
 \multirow{2}{*}{$n$}  &  \multirow{2}{*}{$k_1$}  &  \multirow{2}{*}{$k_2$}     &\multirow{2}{*}{Status}   &    \multicolumn{2}{|c|}{DDS 3.0} \\   \cline{5-6}
 &     &   &    & Ite & time \\ \hline 
 20 & 10  & 5    & Opt Soln   &    10  &  1.93   \\ \hline
 20 & 10  & 5    & Unbounded   &    14  &  2.86   \\ \hline
 20 & 15  & 10    & Opt Soln   &    8  &  2.25   \\ \hline
 20 & 15  & 10    & Unbounded   &    15  &  4.41   \\ \hline
 50 & 10  & 5    & Opt Soln   &    12  &  6.11   \\ \hline
 50 & 10  & 5    & Unbounded   &    13  &  6.11   \\ \hline
 50 & 25  & 10    & Opt Soln   &    13  &  12.43   \\ \hline
 50 & 25  & 10    & Unbounded   &    14  &  20.37   \\ \hline
 50 & 30  & 20    & Opt Soln   &    11  &  21.44   \\ \hline
 50 & 30  & 20    & Unbounded   &    16  &  28.79   \\ \hline
 100 & 10  & 5    & Opt Soln   &    15  &  19.17   \\ \hline
 100 & 10  & 5    & Unbounded   &    14  &  20.96   \\ \hline
 100 & 20  & 10    & Opt Soln   &    15  &  43.77   \\ \hline
 100 & 20  & 10    & Unbounded   &    14  &  42.14   \\ \hline
 100 & 30  & 20    & Opt Soln   &    16  &  69.95   \\ \hline
 100 & 30  & 20    & Unbounded   &    15  &  77.63   \\ \hline
 200 & 10  & 5    & Opt Soln   &    17  &  57.23   \\ \hline
 200 & 10  & 5    & Unbounded   &    14  &  58.52   \\ \hline
 200 & 20  & 10    & Opt Soln   &    19  &  126.01   \\ \hline
 200 & 20  & 10    & Unbounded   &    14  &  115.80   \\ \hline
 200 & 30  & 20    & Opt Soln   &    18  &  203.80   \\ \hline
 200 & 30  & 20    & Unbounded   &    16  &  187.96   \\ \hline
 500 & 10  & 5    & Opt Soln   &    19  &  241.28   \\ \hline
 500 & 10  & 5    & Unbounded   &    14  &  232.18   \\ \hline
 500 & 20  & 10    & Opt Soln   &    20  &  617.93   \\ \hline
 500 & 20  & 10    & Unbounded   &    14  &  479.62   \\ \hline
 500 & 30  & 20    & Opt Soln   &    21  &  1007.14   \\ \hline
 500 & 30  & 20    & Unbounded   &    15  &  895.54   \\ \hline

\end{tabular}
\end{table}

\subsection{Hyperbolic Programming combined with other classes}
Although the primary improvements in the manuscript pertain to the treatment of hyperbolic polynomials represented via straight-line programs (SLPs), the updated solver DDS 3.0 demonstrates general enhancements over DDS 2.2, even for problems represented by monomials. These improvements are largely due to the introduction of a new corrector step, which eliminates the need for tuning at every iteration.
Table \ref{table:HB} includes benchmark problems from our library that were also reported in earlier versions of DDS and are now solved using DDS 3.0. These are a set of problems with combinations of hyperbolic polynomial inequalities and other types of constraints such as those arising from LP, SOCP, and entropy function. As shown, DDS 3.0 achieves at least a twofold reduction in runtime for these instances.  

\begin{table} [!ht] 
  \caption{Results for problems involving hyperbolic polynomials using DDS 2.2 and DDS 3.0. The times are in seconds.}
  \label{table:HB}
  \renewcommand*{\arraystretch}{1.3}
  \begin{tabular}{ |c | c| c | c | c |  c| c|c|}
    \hline
\multirow{2}{*}{Problem}  &    \multirow{2}{*}{size of $A$}   &    \multirow{2}{*}{var/deg of $p(z)$} & \multirow{2}{*}{Type of Constraints}   &    Iter/time &   Iter/time \\ 

&&&& DDS 2.2& DDS 3.0 \\ \hline 
%Vamos-1  & $8*4$  & 8/4  &  HB & 7/ \ 0.7    \\ \hline
%Vamos-LP-1  & $12*4$  & 8/4  &  HB-LP & 8/ \ 0.3    \\ \hline
%Vamos-SOCP-1  & $17*5$  & 8/4  &  HB-LP-SOCP & 11/ \ 0.5    \\ \hline
%Vamos-Entr-1  & $17*5$  & 8/4  &  HB-LP-Entropy & 9/ \ 1.49    \\ \hline
VL-Entr-1  & $41*11$  & 20/4  &  HB-LP-Entropy & 11/\ 7.2 & 10/ \ 4.4   \\ \hline
VL-Entr-2  & $61*16$  & 30/4  &  HB-LP-Entropy & 12/\ 60.2  & 10/ \ 27.8  \\ \hline
HPLin-Entr-1  & $111*51$  & 10/15  &  HB-LP-Entropy & 19/\ 12.8 & 15/\ 3.5    \\ \hline
HPLinDer-Entr-1  & $111*51$  & 10/15  &  HB-LP-Entropy & 18 /\ 36.2  & 14/\ 12.8   \\ \hline
HPLinDer-pn-1  & $111*51$  & 10/15  &  HB-LP-pnorm & 32/61.6 & 36/\ 34.6    \\ \hline
Elem-Entr-10-2  & $221*101$  & 10/4  &  HB-LP-Entropy & 18/\ 11.3 & 17 /\ 5.6   \\ \hline
\end{tabular}
\end{table}

\section{Conclusion}
In this manuscript, we developed efficient implementations of interior-point method modules for solving hyperbolic programming problems where the defining polynomial is represented using a straight-line program (SLP). We introduced several classes of hyperbolic polynomials that can be efficiently expressed in SLP form and used these formulations to construct a comprehensive benchmark library. The proposed methods have been integrated into the latest release of our optimization package, DDS 3.0, substantially improving its performance and scalability. DDS 3.0 can now handle practically sized hyperbolic programming problems and can be combined with other convex programming classes for broader applicability.
This work is expected to stimulate renewed interest in developing efficient interior-point methods for hyperbolic programming and comparing them with other methods such as first-order methods. Beyond the improvements introduced in this study, future efforts will focus on expanding the benchmark library with additional families of hyperbolic polynomials that admit efficient straight-line program (SLP) representations. We also aim to explore further applications and theoretical directions arising from hyperbolic programming.

\renewcommand{\baselinestretch}{1}
\bibliographystyle{siam}
\bibliography{References}

\begin{thebibliography}{10}

\bibitem{amini2019}
{\sc N.~Amini}, {\em Spectrahedrality of hyperbolicity cones of multivariate
  matching polynomials}, Journal of Algebraic Combinatorics, 50 (2019),
  pp.~165--190.

\bibitem{AminiBranden2018}
{\sc N.~Amini and P.~Br\"{a}nd\'{e}n}, {\em Non-representable hyperbolic
  matroids}, Adv. Math., 334 (2018), pp.~417--449.

\bibitem{AOV2018}
{\sc N.~Anari, S.~Oveis~Gharan, and C.~Vinzant}, {\em Log-concave polynomials,
  entropy, and a deterministic approximation algorithm for counting bases of
  matroids}, in 59th {A}nnual {IEEE} {S}ymposium on {F}oundations of {C}omputer
  {S}cience---{FOCS} 2018, IEEE Computer Soc., Los Alamitos, CA, 2018,
  pp.~35--46.

\bibitem{bauschke2001hyperbolic}
{\sc H.~H. Bauschke, O.~G{\"u}ler, A.~S. Lewis, and H.~S. Sendov}, {\em
  Hyperbolic polynomials and convex analysis}, Canadian Journal of Mathematics,
  53 (2001), pp.~470--488.

\bibitem{borcea2009lee}
{\sc J.~Borcea and P.~Br{\"a}nd{\'e}n}, {\em The {L}ee-{Y}ang and
  {P}{\'o}lya-{S}chur programs. {I}. linear operators preserving stability},
  Inventiones mathematicae, 177 (2009), pp.~541--569.

\bibitem{branden2007polynomials}
{\sc P.~Br{\"a}nd{\'e}n}, {\em Polynomials with the half-plane property and
  matroid theory}, Advances in Mathematics, 216 (2007), pp.~302--320.

\bibitem{branden2011obstructions}
\leavevmode\vrule height 2pt depth -1.6pt width 23pt, {\em Obstructions to
  determinantal representability}, Advances in Mathematics, 226 (2011),
  pp.~1202--1212.

\bibitem{branden2014hyperbolicity}
\leavevmode\vrule height 2pt depth -1.6pt width 23pt, {\em Hyperbolicity cones
  of elementary symmetric polynomials are spectrahedral}, Optimization Letters,
  8 (2014), pp.~1773--1782.

\bibitem{burton2014real}
{\sc S.~Burton, C.~Vinzant, and Y.~Youm}, {\em A real stable extension of the
  {V}amos matroid polynomial}, arXiv preprint arXiv:1411.2038,  (2014).

\bibitem{choe2004homogeneous}
{\sc Y.-B. Choe, J.~G. Oxley, A.~D. Sokal, and D.~G. Wagner}, {\em Homogeneous
  multivariate polynomials with the half-plane property}, Advances in Applied
  Mathematics, 32 (2004), pp.~88--187.

\bibitem{coey2022solving}
{\sc C.~Coey, L.~Kapelevich, and J.~P. Vielma}, {\em Solving natural conic
  formulations with hypatia. jl}, INFORMS Journal on Computing, 34 (2022),
  pp.~2686--2699.

\bibitem{DTV2025}
{\sc J.~Dahl, L.~Tun{\c{c}}el, and L.~Vandenberghe}, {\em New complexity bounds
  for primal--dual interior-point algorithms in conic optimization}, arXiv
  preprint arXiv:2509.10263,  (2025).

\bibitem{goodfellow2016deep}
{\sc I.~Goodfellow, Y.~Bengio, A.~Courville, and Y.~Bengio}, {\em Deep
  learning}, vol.~1, MIT press Cambridge, 2016.

\bibitem{Clarabel_2024}
{\sc P.~J. Goulart and Y.~Chen}, {\em Clarabel: An interior-point solver for
  conic programs with quadratic objectives}, 2024.

\bibitem{griewank2008evaluating}
{\sc A.~Griewank and A.~Walther}, {\em Evaluating derivatives: principles and
  techniques of algorithmic differentiation}, SIAM, 2008.

\bibitem{guler1997hyperbolic}
{\sc O.~G{\"u}ler}, {\em Hyperbolic polynomials and interior point methods for
  convex programming}, Mathematics of Operations Research, 22 (1997),
  pp.~350--377.

\bibitem{gurvits2006hyperbolic}
{\sc L.~Gurvits}, {\em Hyperbolic polynomials approach to {V}an der
  {W}aerden/{S}chrijver-{V}aliant like conjectures: sharper bounds, simpler
  proofs and algorithmic applications}, in Proceedings of the thirty-eighth
  annual ACM symposium on Theory of computing, 2006, pp.~417--426.

\bibitem{Gurvits2008}
\leavevmode\vrule height 2pt depth -1.6pt width 23pt, {\em Van der
  {W}aerden/{S}chrijver-{V}aliant like conjectures and stable (aka hyperbolic)
  homogeneous polynomials: one theorem for all}, Electron. J. Combin., 15
  (2008), pp.~Research Paper 66, 26.
\newblock With a corrigendum.

\bibitem{HV2007}
{\sc J.~W. Helton and V.~Vinnikov}, {\em Linear matrix inequality
  representation of sets}, Comm. Pure Appl. Math., 60 (2007), pp.~654--674.

\bibitem{helton2007linear}
{\sc J.~W. Helton and V.~Vinnikov}, {\em Linear matrix inequality
  representation of sets}, Communications on Pure and Applied Mathematics: A
  Journal Issued by the Courant Institute of Mathematical Sciences, 60 (2007),
  pp.~654--674.

\bibitem{DDS3.0}
{\sc M.~Karimi and L.~Tun\c{c}el}, {\em {DDS (Domain-Driven Solver) version
  3.0}}.
\newblock Available at \url{https://github.com/mehdi-karimi-math/DDS}.

\bibitem{Karimi_Library_for_Modern}
\leavevmode\vrule height 2pt depth -1.6pt width 23pt, {\em {Library for Modern
  Convex Optimization}}.
\newblock Available at
  \url{https://github.com/mehdi-karimi-math/Library-for-Modern-Convex-Optimization}.

\bibitem{DDS}
\leavevmode\vrule height 2pt depth -1.6pt width 23pt, {\em Domain-{D}riven
  {S}olver ({DDS}) {V}ersion 2.1: {A} {MATLAB}-based software package for
  convex optimization problems in domain-driven form}, Math. Program. Comput.,
  16 (2024), pp.~37--92.

\bibitem{karimi2020primal}
{\sc M.~Karimi and L.~Tun{\c{c}}el}, {\em Primal--dual interior-point methods
  for domain-driven formulations}, Mathematics of Operations Research, 45
  (2020), pp.~591--621.

\bibitem{karimi_status_arxiv}
\leavevmode\vrule height 2pt depth -1.6pt width 23pt, {\em Status determination
  by interior-point methods for convex optimization problems in
  {D}omain-{D}riven form}, Mathematical Programming, 194 (2022), pp.~937--974.

\bibitem{karimi2025efficient}
{\sc M.~Karimi and L.~Tuncel}, {\em Efficient implementation of interior-point
  methods for quantum relative entropy}, INFORMS Journal on Computing, 37
  (2025), pp.~3--21.

\bibitem{klingler2025homotopy}
{\sc A.~Klingler and T.~Netzer}, {\em Homotopy methods for convex
  optimization}, SIAM Journal on Optimization, 35 (2025), pp.~1498--1523.

\bibitem{K2021}
{\sc M.~Kummer}, {\em Spectral linear matrix inequalities}, Adv. Math., 384
  (2021), pp.~Paper No. 107749, 36.

\bibitem{LPR2005}
{\sc A.~S. Lewis, P.~A. Parrilo, and M.~V. Ramana}, {\em The {L}ax conjecture
  is true}, Proc. Amer. Math. Soc., 133 (2005), pp.~2495--2499.

\bibitem{MSS2015}
{\sc A.~W. Marcus, D.~A. Spielman, and N.~Srivastava}, {\em Interlacing
  families {II}: {M}ixed characteristic polynomials and the {K}adison-{S}inger
  problem}, Ann. of Math. (2), 182 (2015), pp.~327--350.

\bibitem{mosek}
{\sc {MOSEK ApS}}, {\em The MOSEK optimization toolbox for MATLAB manual.
  Version 9.0.}, 2019.
\newblock \url{http://docs.mosek.com/9.0/toolbox/index.html}.

\bibitem{Myklebust2015}
{\sc T.~G.~J. Myklebust}, {\em On primal-dual interior-point algorithms for
  convex optimisation}, PhD thesis, Department of Combinatorics and
  Optimization, Faculty of Mathematics, University of Waterloo.

\bibitem{nagano2024projection}
{\sc T.~Nagano, B.~F. Louren{\c{c}}o, and A.~Takeda}, {\em Projection onto
  hyperbolicity cones and beyond: a dual frank-wolfe approach}, arXiv preprint
  arXiv:2407.09213,  (2024).

\bibitem{interior-book}
{\sc Y.~Nesterov and A.~Nemirovski}, {\em Interior-Point Polynomial Algorithms
  in Convex Programming}, SIAM Series in Applied Mathematics, SIAM:
  Philadelphia, 1994.

\bibitem{NT2016}
{\sc Y.~Nesterov and L.~Tun{\c{c}}el}, {\em Local superlinear convergence of
  polynomial-time interior-point methods for hyperbolicity cone optimization
  problems}, SIAM J. Optim., 26 (2016), pp.~139--170.

\bibitem{NS2015}
{\sc T.~Netzer and R.~Sanyal}, {\em Smooth hyperbolicity cones are
  spectrahedral shadows}, Math. Program., 153 (2015), pp.~213--221.

\bibitem{Oliveira2020}
{\sc R.~Oliveira}, {\em Conditional lower bounds on the spectrahedral
  representation of explicit hyperbolicity cones}, in Proceedings of the 45th
  International Symposium on Symbolic and Algebraic Computation, 2020,
  pp.~396--401.

\bibitem{alfonso}
{\sc D.~Papp and S.~Y{\i}ld{\i}z}, {\em {A}lfonso: Matlab package for
  nonsymmetric conic optimization}, INFORMS Journal on Computing,  (2021).

\bibitem{Petrovsky1937}
{\sc I.~G. Petrovsky}, {\em {\"U}ber das {C}auchysche {P}roblem f{\"u}r
  {S}ysteme von partiellen {D}ifferentialgleichungen}, Recueil Math{\'e}matique
  (Matematicheskii Sbornik), 2 (44) (1937), pp.~815--870.

\bibitem{RRSW2019}
{\sc P.~Raghavendra, N.~Ryder, N.~Srivastava, and B.~Weitz}, {\em Exponential
  lower bounds on spectrahedral representations of hyperbolicity cones}, in
  Proceedings of the Thirtieth Annual ACM-SIAM Symposium on Discrete
  Algorithms, SIAM, 2019, pp.~2322--2332.

\bibitem{renegar2006hyperbolic}
{\sc J.~Renegar}, {\em Hyperbolic programs, and their derivative relaxations},
  Foundations of Computational Mathematics, 6 (2006), pp.~59--79.

\bibitem{saunderson2015polynomial}
{\sc J.~Saunderson and P.~A. Parrilo}, {\em Polynomial-sized semidefinite
  representations of derivative relaxations of spectrahedral cones},
  Mathematical Programming, 153 (2015), pp.~309--331.

\bibitem{SV2018}
{\sc E.~Shamovich and V.~Vinnikov}, {\em Livsic-type determinantal
  representations and hyperbolicity}, Adv. Math., 329 (2018), pp.~487--522.

\bibitem{sedumi}
{\sc J.~F. Sturm}, {\em Using {SeDuMi} 1.02, a {MATLAB} toolbox for
  optimization over symmetric cones}, Optimization Methods and Software, 11
  (1999), pp.~625--653.

\bibitem{toh1999sdpt3}
{\sc K.-C. Toh, M.~J. Todd, and R.~H. T{\"u}t{\"u}nc{\"u}}, {\em {SDPT3}-- a
  {MATLAB} software package for semidefinite programming, version 1.3},
  Optimization Methods and Software, 11 (1999), pp.~545--581.

\bibitem{Vinnikov1993}
{\sc V.~Vinnikov}, {\em Selfadjoint determinantal representations of real plane
  curves}, Math. Ann., 296 (1993), pp.~453--479.

\bibitem{vinnikov2012lmi}
\leavevmode\vrule height 2pt depth -1.6pt width 23pt, {\em Lmi representations
  of convex semialgebraic sets and determinantal representations of algebraic
  hypersurfaces: past, present, and future}, in Mathematical Methods in
  Systems, Optimization, and Control: Festschrift in Honor of J. William
  Helton, Springer, 2012, pp.~325--349.

\bibitem{Wagner2011}
{\sc D.~G. Wagner}, {\em Multivariate stable polynomials: theory and
  applications}, Bull. Amer. Math. Soc. (N.S.), 48 (2011), pp.~53--84.

\bibitem{wagner2009criterion}
{\sc D.~G. Wagner and Y.~Wei}, {\em A criterion for the half-plane property},
  Discrete Mathematics, 309 (2009), pp.~1385--1390.

\bibitem{Zinchenko2008}
{\sc Y.~Zinchenko}, {\em On hyperbolicity cones associated with elementary
  symmetric polynomials}, Optim. Lett., 2 (2008), pp.~389--402.

\end{thebibliography}

\appendix

\end{document}